\documentclass{amsart}
\usepackage{hyperref}

\usepackage{amsrefs,amsthm,amsmath,amsfonts,amssymb,graphicx,latexsym,color}

\newtheorem{theorem}{Theorem}
\newtheorem*{assumeU}{Assumption U}

\newtheorem*{assumeL}{Assumption L}

\newtheorem*{corollary}{Corollary L}
\newtheorem{lemma}{Lemma}

\newtheorem{proposition}[lemma]{Proposition}

\theoremstyle{definition}

\newcommand{\R}{{\mathbb R}}

\newcommand{\ang}[1]{ \left< {#1} \right> }
\newcommand{\ext}[1]{ \underline{ {#1} } }
\newcommand{\ind}{ {\mathbf 1}}
\newcommand{\sph}{{\mathbb S}^{n-1}}
\newcommand{\nsm}{|}

\newcommand{\set}[2]{ \left\{ #1 \ \left| \ #2 \right. \right\} }
\newcommand{\ip}[2]{ \left< #1 , #2 \right>}

\newcommand{\last}{{n+1}}
\newcommand{\ddim}{{n}}
\newcommand{\threed}{{\mathbb R}^n}
\newcommand{\eqdef}{\overset{\mbox{\tiny{def}}}{=}}
\newcommand{\opGplus}{\mathcal{D}^k_{+}}
\newcommand{\opGminus}{\mathcal{D}^k_{-}}
\newcommand{\opGstar}{\mathcal{D}^k_{*}}
\newcommand{\opGstarS}{\mathcal{O}_{*}}
\newcommand{\collOPd}{\mathcal{Q}_g}
\newcommand{\ScollOPd}{\mathcal{Q}_g^*}
\newcommand{\Ctheta}{\vartheta}
\newcommand{\Cupper}{C_g}
\newcommand{\Clower}{\tilde{C}_g}
\newcommand{\CurveP}{\zeta}
\newcommand{\CurveEP}{\ext{\zeta}}
\newcommand{\dsp}{{\mathbb S}}
\newcommand{\ballR}{{B_R}}
\newcommand{\nspace}{N^{s,\gamma}}
\newcommand{\AGthing}{A_{g}^{G}}

\newcommand{\domain}{\Omega}

\begin{document}

\title[Sharp anisotropic estimates for the Boltzmann collision operator]{Sharp anisotropic estimates for the Boltzmann collision operator and its entropy production}

\author[P. T. Gressman]{Philip T. Gressman}
\address{University of Pennsylvania, Department of Mathematics, David Rittenhouse Lab, 209 South 33rd Street, Philadelphia, PA 19104-6395, USA} 
\email{gressman at math.upenn.edu \& strain at math.upenn.edu}
\urladdr{http://www.math.upenn.edu/~gressman/ \& http://www.math.upenn.edu/~strain/}
\thanks{P.T.G. was partially supported by the NSF grant DMS-0850791, and an Alfred P. Sloan Foundation Research Fellowship.}

\author[R. M. Strain]{Robert M. Strain}
\thanks{R.M.S. was partially supported by the NSF grant DMS-0901463, and an Alfred P. Sloan Foundation Research Fellowship.}


\begin{abstract}
This article provides sharp constructive upper and lower bound estimates for the  
Boltzmann collision operator 
with the full range of physical non cut-off collision kernels ($\gamma > -n$ and $s\in (0,1)$)
 in the trilinear $L^2(\threed)$ energy $\langle \mathcal{Q}(g,f),f\rangle$.  These new estimates prove that, for a very general class of $g(v)$, the  global diffusive behavior (on $f$) in the energy space is that of the geometric fractional derivative semi-norm identified in the linearized context in our earlier works \cite{gsNonCutJAMS,gsNonCutA}.  We further prove new global entropy production estimates with the same anisotropic semi-norm.  
 This resolves the longstanding, widespread heuristic conjecture about the sharp diffusive nature of the non cut-off Boltzmann collision operator in the energy space $L^2(\threed)$.
\end{abstract}
\maketitle

\thispagestyle{empty}

\tableofcontents

\section{Introduction and main results}

Our motivation for this study 
is derived from  the physical {\em Boltzmann equation},  
$$
  \frac{\partial f}{\partial t} + v \cdot \nabla_x f = {\mathcal Q}(f,f),
$$
where the unknown $f(t,x,v)$ is a nonnegative function.  For each time $t\ge 0$, the solution $f(t, \cdot, \cdot)$ represents the empirical measure of particles. 
The spatial coordinates are generally $x\in \domain$ and the velocities are $v\in\threed$, where $\domain \subset \threed$ is a domain $(\ddim\ge2)$.  The Boltzmann equation, derived in 1872, is one of the fundamental equations of mathematical physics and, in particular,
 a cornerstone of statistical physics.

In this article, we prove new sharp, anisotropic fractional derivative estimates for the
  {\em Boltzmann collision operator} ${\mathcal Q}$.  
This is a bilinear operator which is given by
  \begin{equation}
  {\mathcal Q} (g,f)(v) \eqdef 
  \int_{\threed}  dv_* 
  \int_{\sph}  d\sigma~ 
  B(v-v_*, \sigma) \, 
  \big[ g'_* f' - g_* f \big].
  \label{BoltzCOL}
  \end{equation}
We are using the standard shorthand $f = f(v)$, $g_* = g(v_*)$, $f' = f(v')$, 
$g_*^{\prime} = g(v'_*)$. 
In this expression, $v$, $v_*$  and  $v'$, $v' _*$ are 
the velocities in $\threed$ of a pair of particles before and after collision.  They are connected through the formulas
  \begin{equation}
  v' = \frac{v+v_*}{2} + \frac{|v-v_*|}{2} \sigma, \qquad
  v'_* = \frac{v+v_*}{2} - \frac{|v-v_*|}{2} \sigma,
  \qquad \sigma \in \sph.
  \label{sigma}
  \end{equation}
  These formulas correspond physically to elastic collisions with conserved quantities
  $$
v+v_* = v^{\prime } + v^{\prime }_*,
\quad
|v|^2+|v_*|^2 = |v^\prime|^2+|v_*^\prime|^2.
$$
Our main focus is to study the sharp, anisotropic fractional diffusive effects induced by this operator under fully general physical assumptions on the collision kernel.

 The {\em Boltzmann collision kernel} $B(v-v_*, \sigma)$ for a monatomic gas is a nonnegative function which 
only depends on the {\em relative velocity} $|v-v_*|$ and on
the {\em deviation angle}  $\theta$ through 
$\cos \theta = \ip{ k }{ \sigma}$ where $k = (v-v_*)/|v-v_*|$ and $\langle \cdot, \cdot \rangle$ is the usual scalar product in $\threed$. 
Without loss of generality we may assume that  $B(v-v_*, \sigma)$
is supported on $\ip{ k }{ \sigma} \ge 0$, i.e. $0 \le \theta \le \frac{\pi}{2}$.
Otherwise we can  reduce to this situation 
with the following customary ``symmetrization'': 
$$
\overline{B}(v-v_*, \sigma)
=
\left[ 
B(v-v_*, \sigma)
+
B(v-v_*, -\sigma)
\right]
{\bf 1}_{\ip{ k }{ \sigma} \ge 0}.
$$ 
Above and generally, ${\bf 1}_{A}$ is the usual indicator function of the set $A$.

\subsection*{The Collision Kernel}

Our assumptions are as follows:  
 \begin{itemize}
 \item 
We suppose that $B(v-v_*, \sigma)$ takes product form in its arguments as
\begin{equation}
B(v-v_*, \sigma) =\Phi( |v-v_*| ) \, b(\cos \theta).
\notag
\end{equation}
In general, both $b$ and $\Phi$ are non-negative functions. 
 
  \item The angular function $t \mapsto b(t)$ is not locally integrable; for $c_b >0$ it satisfies
\begin{equation}
\frac{c_b}{\theta^{1+2s}} 
\le 
\sin^{n-2} \theta \  b(\cos \theta) 
\le 
\frac{1}{c_b\theta^{1+2s}},
\quad
s \in (0,1),
\quad
   \forall \, \theta \in \left(0,\frac{\pi}{2} \right].
   \label{kernelQ}
\end{equation}
   Some authors use the notation 
$\nu = 2s \in (0,2)$, which is equivalent. 
 
  \item The kinetic factor $z \mapsto \Phi(|z|)$ satisfies for some $C_\Phi >0$
\begin{equation}
\Phi( |v-v_*| ) =  C_\Phi  |v-v_*|^\gamma , \quad \gamma  > -n.
\label{kernelP}
\end{equation}
 \end{itemize}

Our main physical motivation is derived from  particles interacting according to a spherical intermolecular repulsive potential of the form
$$
\phi(r)=r^{-(p-1)} , \quad p \in (2,+\infty).
$$
For these potentials, Maxwell 
in 1866 showed that 
the kernel $B$ can be estimated.   In dimension $n=3$, $B$ satisfies the conditions above with 
$\gamma = (p-5)/(p-1)$ 
and $s = 1/(p-1)$; 
see for instance \cite{MR1307620,MR1942465}.
Thus the conditions in  \eqref{kernelQ} and \eqref{kernelP} include all of the potentials $p > 2$ in the physical dimension $n=3$ as a particular case.
Notice that the Boltzmann collision operator is not well defined for $p=2$, see \cite{MR1942465}.

The celebrated Boltzmann $H$-theorem is one of the hallmarks of statistical physics. 
We define the H-functional by
$
{H}(t)\eqdef-\int_{\domain}dx~ \int_{\threed}dv~ f\log f
$
for a suitable domain $\domain$.
Then the Boltzmann H-theorem predicts that, for solutions of the Boltzmann equation, the entropy is increasing over time; formally (neglecting the boundary of $\domain$), this corresponds to the statement
$$
\frac{d{H}(t)}{dt} =  \int_{\domain}dx~ D(f,f) \ge 0.
$$
This is a demonstration of the second law of thermodynamics.  Here the entropy production functional, which is nonnegative, is given by
\begin{equation}
\label{entropyPRODf}
D(g,f) \eqdef -\int_{\threed} ~ dv ~ \mathcal{Q}(g,f)\log f.
\end{equation}
Moreover, this functional is zero if and only if it is operating on a Maxwellian equilibrium, e.g. \eqref{maxwellianEQUI}.
This predicts 
that the Boltzmann equation exhibits irreversible dynamics and convergence to Maxwellian in large time.

In our recent works \cite{gsNonCutA,gsNonCutJAMS} on the global-in-time stability of the Boltzmann equation with the physical collision kernels \eqref{kernelQ} and \eqref{kernelP}, and near Maxwellian initial data, we introduced into the Boltzmann theory the 
following anisotropic norm:
\begin{equation} 
\nsm f\nsm_{\nspace}^2 
\eqdef 
\nsm f\nsm_{L^2_{\gamma+2s}}^2 + \int_{\threed} dv ~ \int_{\threed} dv' ~ 
(\ang{v}\ang{v'})^{\frac{\gamma+2s+1}{2}}
~
 \frac{(f' - f)^2}{d(v,v')^{n+2s}} 
\ind_{d(v,v') \leq 1}. 
\notag
\end{equation}
Here we use one of the weighted $L^p_{\ell}$ spaces, $\ell \in \R$, with norm given by
$$
\nsm f\nsm_{L^p_{\ell}}^p 
\eqdef
\int_{\threed} dv ~ 
\ang{v}^{\ell}
~
|f(v)|^p,
\quad 
1\le p < \infty.
$$
The weight is
$
\ang{v}
\eqdef \sqrt{1+|v|^2}.
$
If $\ell = 0$ we will write $\nsm f\nsm_{L^p_{0}} = \nsm f\nsm_{L^p}$.  We also record here the ``dotted''  semi-norm
 \begin{equation} 
\nsm f\nsm_{\dot{N}^{s,\gamma}}^2 
\eqdef 
\int_{\threed} dv ~ \int_{\threed} dv' ~ 
(\ang{v}\ang{v'})^{\frac{\gamma+2s+1}{2}}
~
 \frac{(f(v') - f(v))^2}{d(v,v')^{n+2s}} 
\ind_{d(v,v') \leq 1}. \label{normDOTdef} 
\end{equation}
The fractional differentiation effects are measured 
using the following anisotropic  metric $d(v,v')$ on the ``lifted'' paraboloid:
$$
d(v,v') \eqdef \sqrt{ |v-v'|^2 + \frac{1}{4}\left( |v|^2 -  |v'|^2\right)^2}.
$$
The inclusion of the quadratic difference $|v|^2 - |v'|^2$ is an essential component of the anisotropic fractional differentiation effects induced by the Boltzmann collision operator; it is not a lower order term as we will see in the following.
Heuristically, this metric encodes the anisotropic changes in the power of the weight, which are non-locally
entangled with the  fractional differentiation effects.

We have shown in \cite[Section 2.3]{gsNonCutJAMS} that $\nspace$ sharply characterizes the Dirichlet form of the linearized collision operator.  
In this work, we will  show, perhaps  more interestingly, that the anisotropic diffusive semi-norm \eqref{normDOTdef} sharply characterizes the diffusive effects of the trilinear energy of the non-linear collision operator \eqref{BoltzCOL} under general conditions.  
We will furthermore prove that the diffusive effects of the entropy production functional \eqref{entropyPRODf} are also globally coercively controlled by \eqref{normDOTdef}.
To proceed further, we denote the $L^2(\threed)$ inner product by $\ang{\cdot, \cdot }$, 
which because of the context should not be confused with the scalar product in $\threed$.

It is well known that the Boltzmann collision operator \eqref{BoltzCOL} exhibits parametrized diffusion.   
For this reason we study the operator $\collOPd$, which is defined such that
   \begin{multline*}
\ang{ \mathcal{Q}(g,f), h }
=
  \int_{\threed}  dv
  \int_{\threed}  dv_* 
  \int_{\mathbb{S}^{n-1}}  d\sigma~ 
  B(v-v_*, \sigma) \, 
  \big[ g'_* f' - g_* f \big]~ h
\\
=
  \int_{\threed}  dv
  \int_{\threed}  dv_* 
  \int_{\mathbb{S}^{n-1}}  d\sigma~ 
  B(v-v_*, \sigma) ~ g_* f~  (h' - h)
\eqdef
\ang{ f, \collOPd h }.
  \end{multline*}
This follows from the pre-post collisional change of variables $(v, v_*,\sigma) \to  (v', v_*',k)$, with unit Jacobian, which is standard \cite{MR1942465}.
We recall the widely used decomposition (from \cite{MR1715411}, \cite{MR1765272})
$\ang{ f, \collOPd f } = -N_g(f) + K_g(f)$
where
   \begin{equation}
   \label{cancellationSPLIT}
    \begin{split}
N_g(f)
&\eqdef
  \frac{1}{2} \int_{\threed}  dv
  \int_{\threed}  dv_* 
  \int_{\mathbb{S}^{n-1}}  d\sigma~ 
  B(v-v_*, \sigma) ~ g_* ~ (f' - f)^2,
  \\
K_g(f)
&\eqdef
 \frac{1}{2} \int_{\threed}  dv
  \int_{\threed}  dv_* 
  \int_{\mathbb{S}^{n-1}}  d\sigma~ 
  B(v-v_*, \sigma) ~ g_* ~ \big[(f')^2 - (f)^2\big].
  \end{split}
  \end{equation}
Now it is known from the cancellation lemma \cite{MR1765272} that the second term $K_g(f)$ does not differentiate at all, as will be seen in Section \ref{sec:ls} below.

In what follows, we make two assumptions on  $g=g(v)\ge 0$ from \eqref{cancellationSPLIT}:

\begin{assumeU}
For the upper bound inequalities below we suppose that $g$ satisfies
\begin{equation}
\int_{\threed} dv_* ~ |v - v_*|^{a} \ang{v_*}^i  \left|  g(v_*) \right| 
\le \Cupper \langle v \rangle^{a},
\quad a \in [\gamma, \gamma+2s],
\quad \exists \Cupper>0.
\label{assumeAupper}
\end{equation}
If $s \in (0, 1/2)$, we take $i =1$, and if $s\in [1/2, 1)$, we 
use the power $i =2$. 
\end{assumeU}

Regarding Assumption U:  
if $a\ge 0$ or if $|v - v_*|$ is replaced by  
the regularized kinetic factor
$
\ang{v - v_*}
$ 
then \eqref{assumeAupper} will automatically hold with $\Cupper \approx | g |_{L^1_{i+|a|}}$.  Alternatively, for any $a>-\ddim$, condition \eqref{assumeAupper} is satisfied, for example, by any bounded function which decays at infinity polynomially faster than order $i+\ddim$ (which will be the case whenever $g$ belongs to any weighted Sobolev space of sufficiently high regularity and sufficiently rapid growth of the weight at infinity).
For the coercive lower bounds in Theorems \ref{mainLOWERthm} and \ref{entTHM} we assume:

\begin{assumeL}
Let $R>\delta>0$ be fixed. Suppose  the nonnegative, measurable function $g$ satisfies
\begin{equation}
\int_{\ballR \setminus T_\delta} dv_* ~ g_* \geq \Clower,
\label{assumeAlower}
\end{equation}
uniformly for some positive constant $\Clower$ where $\ballR$ is the Euclidean ball of radius $R$ centered at the origin and $T_{\delta}$ is any linear tube of radius $\delta$.
\end{assumeL}

Notice that such a $g$ in Assumption L need not belong to $L^1(\threed)$ so long as it is locally integrable.  In fact, one can show that any nonnegative, locally integrable $g$ will satisfy \eqref{assumeAlower} unless it equals zero almost everywhere on $\ballR$.   This is considerably more general than the assumptions used in \cite{MR1765272}, which prove  previously-known entropy production estimates.
It includes functions $g$ which satisfy only ``local conservation laws'' for the Boltzmann equation as the following Corollary shows:

\begin{corollary}
Suppose $g=g(v) \ge 0$ is a function belonging to the spaces $L^1(\ballR)$ and $L \log L(\ballR)$ with non-zero norms.  Then \eqref{assumeAlower} and \eqref{mainlower} hold for a constructive constant $\Clower>0$ depending on $R$ and the  norms just mentioned.
\end{corollary}


Moreover, the Boltzmann H-theorem predicts 
that solutions to the Boltzmann equation converge as $t \to \infty$ to the following Maxwellian equilibrium states: 
\begin{equation}
\mu(\rho, u, T)(v) \eqdef \frac{\rho}{(2\pi T)^{n/2}} \exp\left(- \frac{|u-v|^2}{2T}\right),
\label{maxwellianEQUI}
\end{equation}
where $\rho$, $u$, and $T$ are the density, mean velocity and temperature of the gas respectively.   
Thus it can be expected that for $\rho>0$ and $T<\infty$ our assumptions \eqref{assumeAupper} and \eqref{assumeAlower} will be satisfied for all time by any sufficiently regular solution of the Boltzmann equation which respects the H-theorem.

We are now ready to state our main results.

\begin{theorem}\label{mainLOWERthm}  
(Main coercive inequality)
Suppose that \eqref{assumeAlower} holds.
Then
\begin{equation} 
N_g(f)  
\ge 
C_1  \nsm f\nsm_{\dot{N}^{s,\gamma}}^2, 
\label{mainlower}
\end{equation}
where
$
C_1 
\eqdef 
C_{n,R,\delta} \Clower^2 / |g|_{L^1(\ballR)} 
$
and  $C_{n,R,\delta}>0$ only depends upon $n$, $R$, $\delta$, \eqref{kernelQ} and \eqref{kernelP}.  In particular, \eqref{assumeAupper} and \eqref{assumeAlower}  together imply
\begin{equation}
- \ang{\mathcal{Q}(g, f),f}  
\geq 
C_1  \nsm f \nsm_{\dot{N}^{s,\gamma}}^2 
- C_2 \Cupper \nsm f\nsm^2_{L^2_{\gamma}},
\label{lowerSHARP}
\end{equation}
where all the positive constants, in particular $C_{n,R,\delta}$ 
and $C_2$, are constructive.
\end{theorem}

We also have the following upper bound estimate:

\begin{theorem}
\label{mainTHM}  
(Main upper bound)
Under \eqref{assumeAupper} we have the upper bound estimate
\begin{equation}
\left| \ang{\mathcal{Q}(g, f),  h} \right|
\le
C_3 \Cupper \nsm f \nsm_{\nspace}
\nsm h \nsm_{\nspace}.
\notag
\end{equation}
Here, as above, the constant $C_3>0$  is constructive as well.
\end{theorem}

Collecting the estimates in Theorem \ref{mainLOWERthm} and Theorem \ref{mainTHM} we obtain
$$
C_1 \nsm f \nsm_{\dot{N}^{s,\gamma}}^2 
\le
 - \ang{ f, \collOPd f }
 +
  C_2 \Cupper \nsm f\nsm^2_{L^2_{\gamma}}
\le \left( C_2+C_3 \right) 
\Cupper \nsm f \nsm_{\nspace}^2.
$$
These theorems thus show that the sharp global diffusive behavior of $\collOPd$ in $L^2$ is that of the geometric fractional semi-norm \eqref{normDOTdef}, up to terms which do not differentiate.  
Furthermore we can coercively control the entropy production functional \eqref{entropyPRODf}.

 \begin{theorem}\label{entTHM}
(Entropy production)
Given \eqref{assumeAupper} and \eqref{assumeAlower}, we have the constructive lower bound
$$
D(g,f) \ge
C_1 
\nsm \sqrt{f} \nsm_{\dot{N}^{s,\gamma}}^2 - C_2 \Cupper \nsm f\nsm_{L^1_{\gamma}}.
$$
The constants $C_1$, $C_2$ are same as those appearing in \eqref{lowerSHARP}. 
\end{theorem}

Each of the above anisotropic fractional derivative estimates improves upon previously-known estimates, which were formulated in various local and global isotropic Sobolev spaces  such as those found in
\cite{MR2284553,MR2052786,arXiv:0909.1229v1,MR1765272,krmReview2009,
MR1715411,MR1649477}.  More discussion may be found in Section \ref{sec:history}.

In what follows, we use the notation $A \lesssim B$ to mean that  there exists a finite, positive constant $C$ such that $A \leq C B$ holds uniformly over the functions and summation indices which are present in the inequality (and that the precise magnitude of the constant is unimportant).  The notation $B \gtrsim A$ is equivalent to $A \lesssim B$, and $A \approx B$ means that both $A \lesssim B$ and $B \lesssim A$.

\subsection{Historical remarks}  \label{sec:history}
Cercignani \cite[p.85]{MR0255199} in 1969 (about forty years ago) 
noticed
that the linearized Boltzmann collision operator $L$  in the case of the Maxwell molecules collision kernel $(p=5)$ behaves like a fractional diffusive operator; more precisely, in \cite{MR0255199} it was noticed that the eigenvalues of $L$ grow like $-m^{1/4}$  as $m\to\infty$.  
Over time, this point of view transformed into the following widespread heuristic conjecture on the diffusive behavior of the Boltzmann collision operator \eqref{BoltzCOL}:
$$
f \mapsto \mathcal{Q}(g,f) \sim -(-\Delta_v)^s f + \mbox{lower order terms.}
$$
Note that $(-\Delta_v)^s$ is a flat fractional Laplacian.
See for example \cite{MR1942465,arXiv:0909.1229v1,MR1765272}.
This point of view is well-known to be correct locally \cite{MR1765272} in, for instance, a ball $\ballR\subset \threed$ with $0<R<\infty$.  Our main motivation is to provide a  global, sharp anisotropic correction to this heuristic conjecture 
in terms of energy estimates for the trilinear energy under \eqref{assumeAupper} and \eqref{assumeAlower}.  
We furthermore believe that the information provided by these new estimates in Theorems \ref{mainLOWERthm}, \ref{mainTHM}, and \ref{entTHM} will be quite useful to future work in a wide variety of contexts within the Boltzmann theory.

We also point out that the best possible comparisons of the space $\nsm \cdot \nsm_{\nspace}$ to the weighed isotropic Sobolev spaces
$H^s_\ell (\threed)$ are established by the inequalities
\begin{equation}
\nsm h \nsm_{L^2_{\gamma+2s}(\threed)}^2
+
\nsm h \nsm_{H^s_{\gamma}(\threed)}^2
\lesssim 
\nsm h\nsm_{\nspace}^2 
\lesssim \nsm h \nsm_{H^s_{\gamma+2s}(\threed)}^2.
\label{comareISOn}
\end{equation}
Here 
$
H^s_{\ell}(\threed)
\eqdef
\{ f \in L^2_\ell (\threed) : \nsm f \nsm_{H^s_{\ell}(\threed)}^2 
\eqdef \int_{\threed} dv \ang{v}^\ell \left| (I - \Delta_v)^{s/2} f(v) \right|^2 <\infty
\}
$
is the standard isotropic fractional Sobolev space.
These inequalities can be easily established as in \cite[Eq. (2.15)]{gsNonCutJAMS}.

Historically, anisotropic behavior in the linearized context was noticed as early as Pao \cite{MR0636407} in 1974, where he studied the symbol of the Fourier transform of the linearized collision operator.  Indeed, a key difficulty in the analysis of \cite{MR0636407} was the presence of the cross product of the frequency (derivative) variable and the velocity (weight) variable; Pao's pseduodifferential operator approach required an intricate knowledge of Bessel functions.  
Further studies of the Fourier transform of the Boltzmann collision operator were given,  for example, in  
\cite{MR1763526,MR1128328,MR1765272,MR2052786} and the references therein.  The sharp anisotropic differentiation effects for the Dirichlet form of the linearized Landau collision operator have been studied, for example, by Guo \cite{MR1946444} in 2002, 
and Mouhot-Strain \cite{MR2322149} in 2007.  The Landau operator involves whole derivatives rather than non-local, geometric fractional derivatives.   It is worth discussing briefly the terms ``anisotropic'' and ``isotropic'' because, for example,
 Pao \cite{MR0636407} uses the term ``isotropic'' to describe the effects that we are now calling ``anisotropic.''  
Additionally, the terms anisotropic or non-isotropic are sometimes used loosely to mean there is a gain of weights or a gain of regularity in the relevant estimates, which does also occur in some isotropic situations.   
 We use the term anisotropic in this paper to mean that there is an essential coupling between the non-local ``direction of differentiation'' and the power of the velocity weight at infinity (meaning, for example, that the spaces $H^s_{\gamma}$ and $H^s_{\gamma+2s}$ appearing in \eqref{comareISOn} are sharp but not equal to one another).
This strict notion of anisotropy can be observed in the Boltzmann theory using delicate calculations involving the Fourier transform; see, for example, \cite{MR0636407,MR1763526}.   
It is worth noting that a large number of alternate formulas can be derived for the Fourier transform of the Boltzmann collision operator.
However since the Fourier transform is a 
Euclidean invariant object, 
 it seems to be difficult to use the Fourier transform point-of-view 
 to prove sharp anisotropic energy estimates in the presence of essentially non-Euclidean geometries.

In this paper, we treat the Boltzmann collision operator as a fractional, geometric Laplacian with the geometry of a ``lifted'' paraboloid in $\R^\last$ which was introduced in our previous work \cite{gsNonCutJAMS,gsNonCutA}.  In particular, we do not use the Fourier transform at all.  
With this point of view,  we can apply a generalization of Littlewood-Paley theory, as in Section \ref{sec:upTRI}, to prove our sharp anisotropic upper bound estimates in Theorem \ref{mainTHM}.   Our Littlewood-Paley projections, developed in \cite{gsNonCutJAMS}, are inspired by
the Littlewood-Paley-Stein theory
 of Stein \cite{MR0252961},
the geometric Littlewood-Paley theory of
Klainerman-Rodnianski \cite{MR2221254}, and the original physics representations of \eqref{BoltzCOL} in terms of delta functions on the collisional conservation laws. 

Recently this point of view has played a crucial role in our proof  \cite{gsNonCutJAMS,gsNonCutA,sNonCutOp} of global existence of unique, non-negative,  near-Maxwellian classical solutions to the Boltzmann equation (1872) which exhibit rapid convergence to equilibrium. These results cover the full range of physical collision kernels derived by Maxwell in 1866 from an inverse power intermolecular potential, and they
additionally resolve a conjecture from \cite{MR2322149}.
Notice also \cite{newNonCutAMUXY} for other related results.

The sharpest previously-known estimates for the collision operator are expressed in terms of isotropic Sobolev spaces and correspond to the inequalities
\begin{eqnarray*}
\left| \ang{\mathcal{Q}(g, f),  h} \right|
& \lesssim &
\nsm g \nsm_{L^1_{(\gamma+2s)^+}(\R^3)}
\nsm f \nsm_{H^s_{(\gamma+2s)^+}(\R^3)}
\nsm h \nsm_{H^s
(\R^3)},
\\
- \ang{\mathcal{Q}(g, f),f}  & \geq  & C(g) \nsm f \nsm_{H^s_{\gamma}(\R^3)}^2 - C 
\nsm g \nsm_{L^1_{\max\{\gamma^+, 2- \gamma^+\}}(\R^3)}\nsm f\nsm^2_{L^2_{\gamma^+}(\R^3)} .
\end{eqnarray*}
Here $(\gamma+2s)^+ \eqdef \max\{0, \gamma+2s\}$.
For estimates of this type, see, for example, \cite{MR1715411,MR2284553,MR2052786,arXiv:0909.1229v1,MR1765272,MR1942465}  and the references therein; in particular, these (and other) estimates in terms of isotropic Sobolev spaces are stated in \cite[Theorem 2.1, Theorem 2.6]{arXiv:0909.1229v1}.  Here $C(g)>0$ depends upon the $L^1_{\max\{\gamma^+, 2- \gamma^+\}}(\R^3)$
and $L \log L(\R^3)$ norms of $g$ and also the collision kernel $B$.
These bounds are proved in \cite{arXiv:0909.1229v1} when $\Phi$ in \eqref{kernelP} is replaced by the regularized kinetic factor 
$\tilde{\Phi}(|z|) \eqdef C_{\tilde{\Phi}} \ang{z}^\gamma$ which removes the analytical difficulties at zero and infinity. 
Notice that the weights appearing in these isotropic upper and lower bounds correspond to the best isotropic approximations of $\nspace$ from above and below, as indicated by \eqref{comareISOn}.
Our estimates  in Theorems \ref{mainLOWERthm} and \ref{mainTHM}
then improve upon these isotropic estimates in three ways.
Firstly, we include the physical kinetic factors from \eqref{kernelQ} and \eqref{kernelP}.
Secondly, we include the subtle but important sharp, anisotropic fractional diffusive effects in both the upper and the lower bound.
And third, we only need the local information on $g$ such as that in \eqref{assumeAlower} to prove the lower bounds.

Regarding Theorem \ref{entTHM}, many works study entropy production estimates in the non cut-off regime, as in, for example, \cite{MR1649477,MR1765272,MR1715411,MR2052786}. 
These estimates have found widespread utility for both the spatially homogeneous and inhomogeneous Boltzmann equation, see e.g. \cite{MR1942465,MR1765272,MR2052786}.
The most widely-used estimate from \cite{MR1765272} (which is sharp in comparison to ours locally) is  the following:
$$
D(g,f) \ge C_{g,R} ~ |\sqrt{f}|_{\dot{H}^{s}(B_R)}^2 - \mbox{lower order terms.}
$$
The lower order terms above can be found in \cite[Corollary 2]{MR1765272}.  
Notice that the anisotropic effects of \eqref{normDOTdef} are only present near infinity; our estimate in  Theorem \ref{entTHM} thus implies this local estimate under more general conditions on $g$, as in \eqref{assumeAlower}, than those used in \cite{MR1765272}.
Theorem \ref{entTHM} is, moreover, a stronger anisotropic and global version of this local smoothing estimate \cite{MR1765272} (stronger, that is, in terms of the global anisotropic weight  coupled to the fractional differentiation).   We would like to remark that after the submission of the present results, we learned that some isotropic trilinear coercivity estimates were also obtained by Chen and He in 
\cite{ChenHeSmoothing,ChenHeSmoothing2}.

Due to length constraints, it is difficult to provide an exhaustive set of references.  However we refer to further references in our earlier articles \cite{gsNonCutJAMS,gsNonCutA} and the reviews \cite{MR1942465,MR1307620} for more historical discussions of previous results.

\subsection{Outline of the article.}  The remainder of this article is organized as follows.
In Section \ref{sec:ls} we reformulate the Boltzmann collision operator \eqref{BoltzCOL}, using \eqref{cancellationSPLIT} to reduce \eqref{lowerSHARP} to \eqref{mainlower}; we also prove Corollary L.  Then in Section \ref{physicalDECrel} we begin to prove Theorem \ref{mainTHM}; we perform the anisotropic dyadic decomposition of the singularity and subsequently prove the  ``trivial'' estimates as well as the cancellation estimates for the decomposed pieces.  Following that, in Section \ref{sec:upTRI} we describe the generalized anisotropic Littlewood-Paley projections which were developed in our earlier works \cite{gsNonCutA,gsNonCutJAMS}.  We then perform the  triple sum estimates for the trilinear form \eqref{mainexpand}, which completes the proof of Theorem \ref{mainTHM}.
Section \ref{sec:mainCOER} is devoted to the proof of  \eqref{mainlower}.  The methods used here are distinct from those appearing in our earlier works \cite{gsNonCutA,gsNonCutJAMS}.  In particular, we introduce new, non-oscillatory methods to derive a new inequality for $N_g(f)$ which reduces the problem to the study of a family of convolution-like estimates.  The key idea is to observe that $N_g(f)$ and the semi-norm $\dot N^{s,\gamma}$ are already in a functionally similar form, but that the support of the integral appearing in $N_g(f)$ is singular when compared to that of the semi-norm.  By means of the convolution-like operation, we are able to 
``smear-out'' the support of integration in $N_g(f)$ from \eqref{cancellationSPLIT} and thereby make a direct, pointwise comparison.  Then in Section \ref{sec:entropy} we prove Theorem \ref{entTHM} using \eqref{mainlower} and a well-known splitting from, for instance, \cite{MR1715411,MR1765272}.
Lastly, in the Appendix we derive the ``dual formulation'' $\ang{\ScollOPd f,  h }$ for the trilinear form using a Carleman-type representation and methods from \cite[Appendix]{gsNonCutJAMS}.  We will also derive a change of variables that we call the ``co-plane identity''  which is used in Section \ref{sec:regul}.

\subsection{The lower bound estimate \eqref{lowerSHARP} and Corollary L}  \label{sec:ls}

From \eqref{cancellationSPLIT}
and the cancellation lemma \cite[Lemma 1]{MR1765272} we quickly have the identity
$$
K_g(f) = \int_{\threed}  dv_* ~ g_* ~(f^2 * S)(v_*)
=
C'   \int_{\threed}  dv ~f^2  ~ \int_{\threed}  dv_* ~ g_*~ |v-v_*|^\gamma \ge 0,
$$
where $(f^2 * S)$ denotes convolution.
In the formula above, 
\begin{equation}
S(z) \eqdef  \frac{c_n C_\Phi}{2}~ |z|^\gamma \int_0^{\pi/2} d\theta ~ \sin^{n-2} \theta ~ b(\cos\theta) \left[ \cos^{-(\gamma+n)} \frac{\theta}{2} -1\right], 
 \label{est:KGest}
\end{equation}
where now 
$
C' \eqdef \frac{c_n C_\Phi}{2} \int_0^{\pi/2} d\theta ~ \sin^{n-2} \theta ~ b(\cos\theta) \left[ \cos^{-(\gamma+n)} \frac{\theta}{2} -1\right]>0.
$ 
Note that $C'$ is finite by virtue of \eqref{kernelQ},
where $c_n>0$ is a dimensional constant.
By \eqref{assumeAupper} it follows immediately that
$
 K_g(f)  \le  C'  \Cupper \nsm f \nsm_{L^2_{\gamma}}^2.
$
Thus the term   $N_g(f)$ must contain all of the global, geometric fractional differentiation effects.
Subject to \eqref{mainlower}, then, we have shown \eqref{lowerSHARP} from Theorem \ref{mainLOWERthm}.
We will prove \eqref{mainlower} in Section \ref{sec:mainCOER}.


%

Now we will give the proof of Corollary L:

\begin{proof}[Proof of Corollary L]
Fix $R>0$ and let $T_\delta$ be any tube of radius $\delta > 0$.  By Jensen's inequality, we obtain
\[ \left( \frac{\int_{\ballR \cap T_\delta} dv_*  g_*}{|\ballR \cap T_\delta|}  \right) \ln \left( e + \frac{\int_{\ballR \cap T_\delta} dv_*  g_*}{|\ballR \cap T_\delta|} \right) \leq \frac{1}{|\ballR \cap T_\delta|} \int_{\ballR \cap T_\delta} dv_*  g_* \ln (e + g_*), \]
meaning that
\[ 
\left( \frac{\int_{\ballR \cap T_\delta} dv_* ~ g_*}{ \int_{\ballR} dv_* ~  g_* \ln (e + g_*)}  \right) 
\ln \left( e + \frac{\int_{\ballR} dv_*  g_* \ln (e + g_*)}{|\ballR \cap T_\delta|} 
\frac{\int_{\ballR \cap T_\delta} dv_*  g_*}{ \int_{\ballR} dv_* ~  g_* \ln (e + g_*)} \right) \leq 1.
\]
If $|\ballR \cap T_\delta|$ is sufficiently small relative to  $\int_{\ballR} dv_* ~ g_* \ln( e + g_*)$, then this inequality forces $\int_{\ballR \cap T_{\delta}} g_*$ to be arbitrarily small with respect to $\int_{\ballR} dv_* ~ g_* \ln( e + g_*)$.  
In particular, for small enough $\delta$ (keeping $R$ fixed), one can ensure that no more than half of the mass of $g_*$ on the ball of radius $R$ is contained in any tube $T_\delta$:
\[ 
\int_{B_R \setminus T_\delta} dv_* ~ g_* \geq \frac{1}{2} \int_{B_R} dv_* ~ g_*, 
\]
for any tube $T_\delta$ of sufficiently small radius.
\end{proof}

\section{Physical decomposition and individual estimates}
\label{physicalDECrel}

In what follows we prove all of our estimates for functions in the Schwartz space, $\mathcal{S}(\threed)$, which is the well-known space of real-valued $C^{\infty}(\threed)$ functions all of whose derivatives decay at infinity faster than the reciprocal of any polynomial.   
The Schwartz functions are dense in the anisotropic space $\nspace$;  the proof of this fact is easily reduced to the analogous one for Euclidean Sobolev spaces by means of the partition of unity as constructed in \cite[Section 7.3]{gsNonCutJAMS}.  
Moreover, in our estimates, the constants  will not rely on the 
regularity of the functions that we are estimating.  Thus, using standard density arguments, our estimates apply to any function in $\nspace$ or whatever the function space happens to be for a given estimate.  

We will now introduce the anisotropic dyadic decomposition of the singularity \eqref{kernelQ} in Section \ref{sec:diad} and all of the decomposed pieces of the relevant trilinear energy.   Then in Section \ref{sec:trivial} we perform the ``trivial'' estimates of the individual decomposed pieces, and in Section \ref{sec:SSC} we present the cancellation estimates.

\subsection{Dyadic decomposition of the singularity}\label{sec:diad}

Let $\{ \chi_k \}_{k=-\infty}^\infty$ be a partition of unity on $(0,\infty)$ such that $\nsm \chi_k\nsm_{L^\infty} \leq 1$ and 
$\mbox{supp}\left( \chi_k \right)\subset [2^{-k-1},2^{-k}]$.  
For each $k$:
\[
B_k = B_k(v-v_*,\sigma) \eqdef \Phi(|v-v_*|) ~b \left( \left< \frac{v-v_*}{|v-v_*|}, \sigma \right> \right) \chi_k (|v - v'|). \]
Notice that we have the expansion
\begin{equation}
|v-v'|^2 = \frac{|v-v_*|^2}{2} \left( 1 - \left< \frac{v-v_*}{|v-v_*|}, \sigma \right>  \right)
= |v-v_*|^2 \sin^2 \frac{\theta}{2}.
\label{sinEXPest}
\end{equation}
Therefore the condition $|v-v'| \approx 2^{-k}$ is equivalent to the condition that the angle between $\sigma$ and $\frac{v-v_*}{|v-v_*|}$ is comparable to $2^{-k} |v-v_*|^{-1}$.  
With this partition, we define
\begin{equation}
\label{cutoffOP}
\begin{split}  
\opGplus(f,h)  
& \eqdef 
\int_{\threed} dv \int_{\threed} dv_* \int_{\sph} d \sigma ~ B_k(v-v_*, \sigma) ~ g_* f  h',   
\\ 
\opGminus(f,h)  
& \eqdef 
\int_{\threed} dv \int_{\threed} dv_* \int_{\sph} d \sigma ~ B_k(v-v_*, \sigma) ~ g_* f  h.  
\end{split}  
\end{equation}
Notice that both $\opGplus$ and $\opGminus$ depend on $g$ in spite of the fact that the notation used here suppresses this dependence. 

We will also express the collision operator \eqref{BoltzCOL} using the dual representation, denoted by 
$\ScollOPd$.  This is derived via a Carleman-type representation in the Appendix using methods from  \cite[Appendix]{gsNonCutJAMS}.  
With \eqref{3dualZ} in the Appendix we can see that
\begin{equation}
\ang{ f, \collOPd h }
=
\ang{\ScollOPd f,  h }
\eqdef
\int_{\threed} dv'   \int_{\threed} dv_* \int_{E_{v_*}^{v'}} d \pi_{v}  
~\tilde{B} ~ g_* ~ h' \left( f - f' \right)
+
\opGstarS(f,h),
\label{dualOPdef}
\end{equation}
where the kernel $\tilde{B}$ is given by
\begin{equation}
\tilde{B}
\eqdef
2^{n-1}
\frac{B\left(v-v_*, \frac{2v' - v- v_*}{|2v' - v- v_*|}\right) }{  |v'-v_*| ~ |v-v_*|^{n-2}}
\label{kernelTILDE}
\end{equation}
and the operator $\opGstarS$ above  does not differentiate at all:
\begin{equation}
\opGstarS(f,h)  
 \eqdef 
 \int_{\threed} dv' ~ f'  h'   \int_{\threed} dv_*  ~g_*\int_{E_{v_*}^{v'}} d \pi_{v}  ~\tilde{B} 
\left(
 \frac{ \Phi(v'-v_*) |v'-v_*|^{n} }{\Phi(v-v_*) |v-v_*|^{n}}   
 -1
 \right).
 \label{opGlabel}
\end{equation}
We decompose this dual formulation as follows:
\begin{equation}
\label{defTKLcarl}
\begin{split}  
\opGplus(f,h)  
& =
\int_{\threed} dv'   \int_{\threed} dv_* \int_{E_{v_*}^{v'}} d \pi_{v}  
~\tilde{B}_k ~ g_*  h'  f, 
\\ 
\opGstar(f,h)  
& \eqdef 
 \int_{\threed} dv'   \int_{\threed} dv_* \int_{E_{v_*}^{v'}} d \pi_{v}  ~\tilde{B}_k ~ 
  g_* h'  f',
\end{split}  
\end{equation}
where we use the notation 
$$
\tilde{B}_k
\eqdef
2^{n-1}
\frac{B\left(v-v_*, \frac{2v' - v- v_*}{|2v' - v- v_*|}\right) }{  |v'-v_*| ~ |v-v_*|^{n-2}} \chi_k (|v - v'|).
$$
In the integrals \eqref{dualOPdef}, \eqref{opGlabel}, and \eqref{defTKLcarl}, $d\pi_{v}$ is the Lebesgue measure on the $(n-1)$-dimensional plane $E_{v_*}^{v'}$ passing through $v'$ with normal $v' - v_*$, and $v$ is the variable of integration (i.e.,  $E_{v_*}^{v'}$ contains all $v\in \threed$ such that $\ang{v' - v, v' - v_*}=0$). 
Notice the partition of unity guarantees that the kernel is locally integrable for any $k$.

When $f, g, h \in \mathcal{S}(\threed)$, the pre-post collisional change of variables and the dual representation \eqref{dualOPdef}, \eqref{3dualZ} yield the identities
\begin{align*}
 \left<  f, \collOPd h \right>
 =
  \ang{\mathcal{Q}(g,f),  h }
   & = 
\sum_{k=-\infty}^\infty  \left\{ \opGplus(f,h)  - \opGminus(f,h)   \right\}
\\
& = 
\opGstarS(f,h)
+
\sum_{k=-\infty}^\infty  \left\{ \opGplus(f,h)   - \opGstar(f,h)  \right\}. 
\end{align*}
These will be the general quantities that we estimate in the following sections.

\subsection{``Trivial'' estimates of the decomposed pieces} \label{sec:trivial}
The next step is to estimate each of  $\opGplus$, $\opGminus$, $\opGstar$ and $\opGstarS$ using only the known constraints on the size and support of $B_k$ and $\tilde{B}_k$.  
This is what we will do now.

\begin{proposition}
\label{prop11}
Under the assumption \eqref{assumeAupper}, for any integer $k$, 
we have the uniform estimate
\begin{equation}
\left| \opGminus(f,h)   \right|  \lesssim \Cupper ~ 2^{2sk} ~ \nsm f\nsm_{L^2_{\gamma+2s}}  
\nsm h\nsm_{L^2_{\gamma+2s}}. 
\label{tminussmall}
\end{equation}
\end{proposition}

\begin{proof}
Given the size estimates for $b(\cos \theta)$ in \eqref{kernelQ} and the support of $\chi_k$, clearly
\begin{equation}
\int_{\sph} d \sigma ~ B_k \lesssim
\Phi(|v-v_*|)
\int_{2^{-k-1} |v - v_*|^{-1}}^{2^{-k} |v - v_*|^{-1}}  d \theta ~ \theta^{-1-2s} 
\lesssim 2^{2sk} |v-v_*|^{\gamma+2s}.
\label{bjEST}
\end{equation}
Now we use the Cauchy-Schwartz inequality to obtain
\begin{multline}
\label{CSoften}
\left| \opGminus(f,h)  \right| 
\lesssim 2^{2sk} \int_{\threed} dv \int_{\threed} dv_* ~  |v-v_*|^{\gamma+2s} ~|g_*|~ |h f|
\\
\lesssim 2^{2sk} \left(\int_{\threed} dv~ |f|^2 \int_{\threed} dv_* ~ |g_*|~ |v-v_*|^{\gamma+2s}   \right)^{1/2}
\\
\times
\left(\int_{\threed} dv~ |h |^2 \int_{\threed} dv_* ~  |g_*|~ |v-v_*|^{\gamma+2s}   \right)^{1/2}.
\end{multline}
Therefore \eqref{tminussmall} follows from \eqref{assumeAupper}.
\end{proof}

\begin{proposition}
Given \eqref{assumeAupper},
the inequality below is uniform for any integer $k$: 
\begin{equation}
\left| \opGstar(f,h)   \right|  \lesssim \Cupper ~ 2^{2sk} ~ \nsm f\nsm_{L^2_{\gamma+2s}}  
\nsm h\nsm_{L^2_{\gamma+2s}}. \label{tstarsmall}
\end{equation}
\end{proposition}

\begin{proof}
As in the previous proposition, the crucial point in this inequality is the symmetry between $h$ and $f$ combined with Cauchy-Schwartz.  
This time the dual representation \eqref{defTKLcarl} will be used.  
In this case, the key quantity  is
\[ 
\frac{1}{|v'-v_*|} 
\int_{E_{v_*}^{v'}}  d \pi_{v} ~ b \left( \frac{|v'-v_*|^2 - |v - v'|^2}{|v' - v_*|^2 + |v - v'|^2} \right) 
\frac{\chi_k(|v-v'|)}{|v-v_*|^{n-2}}. 
\]
The argument of $b$ in this expression follows from the identity
\begin{equation}
\ang{\frac{v-v_*}{|v-v_*|}, \frac{2v' - v - v_*}{|2v' - v - v_*|} } = \frac{|v' - v_*|^2 - |v - v'|^2}{|v - v'|^2 + |v' - v_*|^2}.
\label{Cidenity}
\end{equation}
The support condition yields $|v-v'| \approx 2^{-k}$.  Moreover, since $b(\cos \theta)$ vanishes for $\theta\in [\pi/2,\pi]$, we have
$|v' - v_*| \ge |v' - v|$. 
Consequently, the condition \eqref{kernelQ} gives 
\begin{equation}
b \left( \frac{|v'-v_*|^2 - |v - v'|^2}{|v' - v_*|^2 + |v - v'|^2} \right)
\lesssim
\left(\frac{ |v - v'|^2}{|v' - v_*|^2} \right)^{-\frac{n-1}{2}-s}.
\notag
\end{equation}
Notice that on $E_{v_*}^{v'}$ we have
$
|v-v_*|^2 = |v-v'|^2 + |v'-v_*|^2 \ge  |v'-v_*|^2.
$
Thus, the integral is bounded by a uniform constant times
\[ 
\int_{E_{v_*}^{v'}}  d \pi_{v} ~\frac{|v'-v_*|^{n-1+2s}}{|v-v'|^{n-1+2s}} |v'-v_*|^{-n+1} \chi_k(|v-v'|) \lesssim 2^{2sk} |v'-v_*|^{2s}. 
\]
By these estimates, it follows that
\begin{equation}  \label{bjCARLest}
\int_{E_{v_*}^{v'}} d \pi_{v}  ~\tilde{B}_k  \lesssim  2^{2sk} |v'-v_*|^{\gamma+2s}.
\end{equation}
As a result, we obtain the upper bound
\[ 
\left| \opGstar(f,h)  \right| \lesssim 2^{2sk} \int_{\threed} dv' \int_{\threed} dv_*   ~ |v' - v_*|^{\gamma+2s}~ |g_* h' f'|. 
\]
Now the estimate \eqref{tstarsmall} easily follows from \eqref{assumeAupper} just as was accomplished in \eqref{CSoften}.
\end{proof}

\begin{proposition}\label{opGstarEST}
Under the assumption \eqref{assumeAupper} with $a = \gamma$,
 we   have 
\begin{equation} 
\left| \opGstarS(f,h)   \right|  \lesssim \Cupper  \nsm f\nsm_{L^2_{\gamma}}  
\nsm h\nsm_{L^2_{\gamma}}.
\notag
\end{equation}
\end{proposition}

\begin{proof}  We recall $\opGstarS(f,h)$ from \eqref{opGlabel}. The key quantity here to estimate is the integral
$
\int_{E_{v_*}^{v'}} d \pi_{v}  ~\tilde{B}_k 
\left(
 A   
 -1
 \right),
$
where
$
A \eqdef
\frac{ \Phi(v'-v_*) |v'-v_*|^{n} }{\Phi(v-v_*) |v-v_*|^{n}}.  
$
Observe that, on $E_{v_*}^{v'}$,
$$
A 
=
 \left( \frac{|v' - v_*|^2}{|v-v' |^2+|v' - v_*|^2}\right)^{\frac{n+\gamma}{2}}.
$$
Now for any fixed $\alpha > 0$, one has
$
\left| c^\alpha - 1 \right|
\lesssim
\left| c - 1 \right|
$
uniformly for $0\le c \le 1$; thus
\begin{align*}
\int_{E_{v_*}^{v'}} d \pi_{v} &  ~\tilde{B}_k 
\left|  A    -   1 \right|
\lesssim
\int_{E_{v_*}^{v'}} d \pi_{v}  ~\tilde{B}_k ~
\frac{|v-v' |^2}{|v-v' |^2+|v' - v_*|^2}
\\
& \lesssim
\int_{E_{v_*}^{v'}} d \pi_{v}  ~\tilde{B}_k ~
\min \left\{1, \frac{|v-v' |^2}{|v' - v_*|^2} \right\}
\\ 
& \lesssim
\min\{ 1, 2^{-2 k} |v' - v_*|^{-2} \}
\int_{E_{v_*}^{v'}} d \pi_{v}  ~\tilde{B}_k
\lesssim
2^{(2s-i) k} |v' - v_*|^{\gamma+2s-i},
\end{align*}
for both $i = 0$ and $i= 2$ (where the last estimate in the series above follows from \eqref{bjCARLest}). We conclude
\begin{multline*} 
\int_{E_{v_*}^{v'}}  d \pi_v ~ \tilde{B} ~ |A-1| 
\lesssim 
 \sum_{k: ~2^k |v' - v_*| \leq 1} 2^{2s k} |v' - v_*|^{\gamma+2s}  
\\
+ 
\sum_{k:~2^k |v' - v_*| > 1} 2^{(2s-2) k} |v' - v_*|^{\gamma+2s-2} 
 \lesssim   |v' - v_*|^{\gamma}. 
 \end{multline*} 
Now we complete the estimate with Cauchy-Schwartz and \eqref{assumeAupper} as in \eqref{CSoften}.
\end{proof}

\begin{proposition}
For any integer $k$, under the assumption \eqref{assumeAupper},
 we have the uniform estimate 
\begin{equation} 
\left| \opGplus(f,h)   \right|  \lesssim \Cupper ~ 2^{2sk} ~ \nsm f\nsm_{L^2_{\gamma+2s}}  
\nsm h\nsm_{L^2_{\gamma+2s}}.
\label{tplussmall}
\end{equation}
\end{proposition}

\begin{proof}
Notice that the operator $\opGplus(f,h)$ is given by  either \eqref{defTKLcarl} (with Carleman variables) or \eqref{cutoffOP} (without Carleman variables).
By Cauchy-Schwartz, we have the inequality
$$
\left| \opGplus(f,h)  \right| 
\le 
\int_{\threed} dv \int_{\threed} dv_* \int_{\sph} d \sigma ~ B_k(v-v_*, \sigma) ~|g_*|~ |f h'|
\le
\sqrt{A_1 A_2}
$$
where
$
A_1 \eqdef
\int_{\threed} dv~ |f|^2 \int_{\threed} dv_* \int_{\sph} d \sigma ~ B_k(v-v_*, \sigma) ~|g_*|
$
and
\begin{multline*}
A_2 \eqdef
\int_{\threed} dv~ |h' |^2 \int_{\threed} dv_* \int_{\sph} d \sigma ~ B_k(v-v_*, \sigma) ~|g_*|
\\
=
\int_{\threed} dv~ |h |^2 \int_{\threed} dv_* \int_{\sph} d \sigma ~ B_k(v-v_*, \sigma) ~|g_*'|
\\
=
\int_{\threed} dv_*~ |h_* |^2 \int_{\threed} dv \int_{\sph} d \sigma ~ B_k(v-v_*, \sigma) ~|g'|.
\end{multline*}
The various equivalent formulas for $A_2$ follow from the pre-post collisional change of variables $(v, v_*,\sigma) \to  (v', v_*',k)$, and the $v \leftrightarrow v_*$, $\sigma \leftrightarrow -\sigma$ symmetry.
Clearly $A_1 \lesssim \Cupper  2^{2sk}  \nsm f\nsm_{L^2_{\gamma+2s}}^2$
as in \eqref{assumeAupper} and \eqref{bjEST}.
%
%
For 
$A_2$, we claim that
\begin{equation}
\label{claimGest}
 \int_{\threed} dv \int_{\sph} d \sigma ~ B_k(v-v_*, \sigma) ~|g'| 
\lesssim 
2^{2sk} \int_{\threed} dv' ~  |g'|~ |v_*-v'|^{\gamma+2s}. 
\end{equation}
Then  $A_2 \lesssim \Cupper 2^{2sk}  \nsm h\nsm_{L^2_{\gamma+2s}}^2$ follows from \eqref{assumeAupper} and the proof will be complete.  With the Carleman representation,  
\cite[Appendix]{gsNonCutJAMS}, the left-hand side of \eqref{claimGest} is
\[ 
\int_{\threed} dv \int_{\sph} d \sigma ~ B_k(v-v_*, \sigma) ~|g'| = \int_{\threed} d v' \int_{E_{v_*}^{v'}} d \pi_{v} ~ \tilde{B}_k ~ |g'|.
\]
Now the claim \eqref{claimGest} follows from the estimate \eqref{bjCARLest}.
\end{proof}

\subsection{Cancellations}
\label{sec:SSC}
In this subsection we study the differences
 $\left( \opGplus   - \opGminus \right)$ and 
$\left( \opGplus   - \opGstar \right)$.
The estimates herein will exploit the cancellations to obtain better dependence on $k$ than in Section \ref{sec:trivial}.  The price to be paid  is that we must measure the magnitude of the differences anisotropically.

The scaling dictated by the problem is that of the paraboloid: namely, that the function $f(v)$ should be thought of as the restriction of some ``lifted'' function $F$ of $n+1$ variables to the paraboloid $(v, \frac{1}{2} |v|^2)$.  Consequently, the correct metric to use in measuring the length of vectors in $\threed$ will be the metric on the paraboloid in $\mathbb{R}^{n+1}$ induced by the $(n+1)$-dimensional Euclidean metric.  To simplify calculations, we  work directly with the function $F$ rather than $f$ and take its $(n+1)$-dimensional derivatives in the usual Euclidean metric.

To begin, we  will write down a formula relating differences of $F$ at nearby points on the paraboloid to  derivatives of $F$ as a function of $n+1$ variables.  
To this end, fix any two $v,v' \in \threed$, and consider $\CurveP : [0,1] \rightarrow \threed$ and $\CurveEP : [0,1] \rightarrow \mathbb{R}^{n+1}$ given by
\[ 
\CurveP(\Ctheta) \eqdef \Ctheta v' + (1-\Ctheta) v,
\quad 
\mbox{ and } 
\quad 
\CurveEP(\Ctheta) \eqdef  \left(\Ctheta v' + (1-\Ctheta)v, \frac{1}{2} \left| \Ctheta v' + (1-\Ctheta) v \right|^2 \right). \]
Here $\CurveEP$ lies in the paraboloid 
$
\set{(v_1,\ldots,v_{n+1}) \in \mathbb{R}^{n+1}}{ v_{n+1} = \frac{1}{2} (v_1^2 + \cdots +v_n^2)}
$;
also note that $\CurveP(0) = v$ and $\CurveP(1) = v'$.  
Elementary calculations show that
$$
\frac{d \CurveEP}{d \Ctheta}(\Ctheta) = \left( v' - v, \ang{\CurveP(\Ctheta), v' - v} \right), 
\quad 
 \mbox{ and }
\quad 
 \frac{d^2 \CurveEP}{d \Ctheta^2} = (0, |v'-v|^2).
$$
Now we use the standard trick of writing the difference of $F$ at two different points in terms of an integral of a derivative (in this case the integral is along the path $\CurveP$):
\begin{align} 
F\left(v',\frac{|v'|^2}{2} \right) - F\left(v, \frac{|v|^2}{2} \right) & = \int_0^1 d \Ctheta ~ \frac{d}{d \Ctheta} F(\CurveEP(\Ctheta)) \nonumber \\
& = \int_0^1 d \Ctheta  \left( \frac{d \CurveEP}{d \Ctheta} \cdot  (\tilde{\nabla} F) (\CurveEP(\Ctheta)) \right), 
\label{paraboladiff}
\end{align}
where the dot product on the right-hand side is the usual Euclidean inner-product on $\mathbb{R}^{n+1}$ and $\tilde{\nabla}$ is the $(n+1)$-dimensional gradient of $F$. 
For convenience we define
\[ 
|\tilde{\nabla}|^i F(v_1,\ldots,v_{n+1}) \eqdef 
\max_{0\le j \leq i}\sup_{|\xi| \leq 1} \left| \left(\xi \cdot \tilde{\nabla} \right)^j F(v_1,\ldots,v_{n+1}) \right|, 
\quad
  i=1,2,
\]
where $\xi \in \mathbb{R}^{n+1}$ and $|\xi|$ is the usual Euclidean length. 


If $v$ and $v'$ are related by the collision geometry, \eqref{sigma}, then $\ang{v-v',v'-v_*} = 0$, which yields that 
$$
\ang{\CurveP(\Ctheta),v'-v}  = \ang{v_*,v'-v} - (1-\Ctheta) |v-v'|^2.
$$
Thus, whenever $|v-v'| \leq 1$, which holds near the singularity ($k\ge 0$), we have
$$
\left| \frac{d \CurveEP}{d \Ctheta} \right| \lesssim |v-v'| \ang{v_*}. 
$$
 In particular, for differences related by the collision geometry we have:
\begin{align}
 \left| F\left(v',\frac{|v'|^2}{2}\right) - F\left(v, \frac{|v|^2}{2}\right)\right| 
 &
  \lesssim 
  \ang{v_*} |v-v'|  \int_0^1  d \Ctheta ~ |\tilde{\nabla}| F (\CurveEP(\Ctheta)). 
 \label{paradiff1} 
\end{align}
Furthermore,
by subtracting the linear term from both sides of \eqref{paraboladiff} and using the integration trick iteratively on the right-hand side of \eqref{paraboladiff}, we obtain
\begin{multline}
\left| F\left(v',\frac{|v'|^2}{2} \right) -  F\left(v, \frac{|v|^2}{2}\right) -  \frac{d \CurveEP}{d \Ctheta}(0) \cdot \tilde{\nabla} F(v) \right| 
\\
 \lesssim
\ang{v_*}^2 |v-v'|^2 \int_0^1 d \Ctheta  ~ | \tilde{\nabla}|^2 F (\CurveEP(\Ctheta)). 
\label{paradiff2}
\end{multline}
We note that, by symmetry, the same result holds when the roles of $v$ and $v'$ are reversed (which only changes the curve $\CurveEP$ by reversing the parametrization: $\CurveEP(\Ctheta)$ becomes $\CurveEP(1-\Ctheta)$).  
We will use these two basic cancellation inequalities to prove the  cancellation estimates for the trilinear form in the following propositions.

\begin{proposition}
\label{cancelFprop}
Suppose $h$ is a Schwartz function on $\threed$ given by the restriction of some Schwartz function $H$ 
on $\mathbb{R}^{n+1}$ to the paraboloid $(v, \frac{1}{2} |v|^2)$.  Let $|\tilde{\nabla}|^i h$ be the restriction of 
$|\tilde{\nabla}|^i H$ to the same paraboloid $(i=1,2)$.  Then, for any $k \geq 0$,
\begin{equation}
\label{cancelf21}
 \left| \left( \opGplus   - \opGminus \right)(f,h) \right| 
   \lesssim 
    \Cupper ~ 
2^{(2s-i)k} ~ \nsm  f\nsm_{L^2_{\gamma+2s}}  
\nsm |\tilde{\nabla}|^i h\nsm_{L^2_{\gamma+2s}}. 
 \end{equation}
 Here when $s\in (0, 1/2)$ in \eqref{kernelQ} then $i=1$ and when $s \in [1/2, 1)$  we have $i = 2$.
\end{proposition}

\begin{proof}
For $s \in [1/2, 1)$,
we write out the relevant difference into two terms
$$
h' -  h
=
\left(  h' -   h  -  \frac{d \CurveEP}{d \Ctheta}(0) \cdot (\tilde{\nabla}   h )(v) \right)
 +
 \frac{d \CurveEP}{d \Ctheta}(0) \cdot (\tilde{\nabla}   h )(v).
$$
We further split 
$
\left(\opGplus   - \opGminus \right)(f,h) 
=
\mathcal{D}^{\mathbb{I}}+ \mathcal{D}^{\mathbb{II}}
$
where $\mathcal{D}^{\mathbb{I}}$ corresponds to the first term in the splitting above.
We begin by considering the last term
$$
\mathcal{D}^{\mathbb{II}}
\eqdef
\int_{\threed} dv \int_{\threed} dv_* \int_{\sph} d \sigma ~B_k(v-v_*, \sigma) ~ g_* f 
~  \frac{d \CurveEP}{d \Ctheta}(0) \cdot (\tilde{\nabla}   h )(v).
$$
Notice that $\frac{d \CurveEP}{d \Ctheta}(0)$
is linear in $v'-v$ and has no other dependence on $v'$.  Thus the symmetry of $B_k$ with respect to $\sigma$ around the direction $\frac{v-v_*}{|v-v_*|}$ forces all components of $v'-v$ to vanish (when integrated in $\sigma$) 
except the component in the symmetry direction.
Thus, one may replace $v'-v$ with $\frac{v-v_*}{|v-v_*|} \ang{v'-v, \frac{v-v_*}{|v-v_*|}}$. 
Since $\ang{v'-v,v'-v_*} = 0$, the vector further reduces to $\frac{v-v_*}{|v-v_*|}\frac{|v'-v|^2}{|v-v_*|}$.  Since $|v'-v| \approx 2^{-k}$ 
we obtain that 
\[ 
\left| \frac{v-v_*}{|v-v_*|}\frac{|v'-v|^2}{|v-v_*|} \right| 
\leq 2^{-2 k} |v-v_*|^{-1}.
\]
The last coordinate direction of $\frac{d \CurveEP}{d \Ctheta}(0)$ is given by $\ang{v,v'-v}$ which
 reduces to 
\[ 
\left| \ang{v, \frac{v-v_*}{|v-v_*|}\frac{|v'-v|^2}{|v-v_*|}} \right| \lesssim 
\left( |v'-v|^2 + |v-v_*|^{-1} |v'-v|^2 \ang{v_*}  \right). 
\]
With these bounds for $\mathcal{D}^{\mathbb{II}}$, we must control the integral
\begin{multline*} 
\left| \mathcal{D}^{\mathbb{II}} \right| 
\lesssim
2^{-2 k} \int_{\threed} dv \int_{\threed} dv_* \int_{\sph} d \sigma ~  B_k ~ \ang{v_*} |g_*| |f| ( |\tilde{\nabla}| h)
\left( 1 + |v-v_*|^{-1} \right)
\\
\lesssim
2^{-2 k} 
\left(
\int_{\threed} dv |f|^2
\int_{\threed} dv_* \ang{v_*} |g_*|  \left( 1 + |v-v_*|^{-1} \right)
\int_{\sph} d \sigma ~  B_k 
\right)^{1/2}
\\
\times
\left(
\int_{\threed} dv  ( |\tilde{\nabla}| h)^2
\int_{\threed} dv_* \ang{v_*} |g_*|  \left( 1 + |v-v_*|^{-1} \right)
\int_{\sph} d \sigma ~  B_k 
\right)^{1/2}.
\end{multline*} 
We complete the estimate for this term with \eqref{bjEST} and  \eqref{assumeAupper} for $a = \gamma+2s$ and $a=\gamma+2s -1$.  Here we used that $2s - 1 \ge 0$ since $s\in [1/2, 1)$.  
(Note that  \eqref{assumeAupper} for the case $a=\gamma+2s -1$ easily follows from the cases $a = \gamma+2s$ and $a = \gamma$.)

We will now estimate $\mathcal{D}^{\mathbb{I}}$.
We use the difference estimate
\eqref{paradiff2}  to obtain that
$$
\left| \mathcal{D}^{\mathbb{I}} \right|
  \lesssim 2^{-2k} \int_0^1 d \Ctheta \int_{\threed} dv \int_{\threed} dv_* \int_{\sph} d \sigma ~ B_k ~ \ang{v_*}^2 ~ |g_* f| 
  ~ |\tilde{\nabla}|^2 h (\CurveP(\Ctheta) ).
$$
Note that the factor $2^{-2k}$ comes directly from \eqref{paradiff2}.
With that last estimate, Cauchy-Schwartz (as in the previous case), and \eqref{assumeAupper}, it suffices to show that
\begin{multline}
 \left( \int_0^1 d \Ctheta \int_{\threed} dv \int_{\threed} dv_* \int_{\sph} d \sigma ~ B_k 
 \ang{v_*}^2  |g_* | 
 \left|  |\tilde{\nabla}|^2 h(\CurveP(\Ctheta))\right|^2 \right)^\frac{1}{2}  
 \\
 \lesssim  
 \Cupper  2^{sk} \nsm   |\tilde{\nabla}|^2 h \nsm_{L^2_{\gamma+2s}}. 
 \label{covterm}
\end{multline}
This  bound
 follows from the  change of variables $u = \Ctheta v' + (1-\Ctheta)v$, which sends  $v$ to $u$.  From the collisional variables \eqref{sigma}, we see  (with $\delta_{ij}$  the Kronecker delta) that 
$$
\frac{d u_i}{ dv_j} = (1-\Ctheta) \delta_{ij} + \Ctheta \frac{d v'_i}{ dv_j}
 = \left(1-\frac{\Ctheta}{2}\right)  \delta_{ij} + \frac{\Ctheta}{2} k_{j} \sigma_{i},
$$
with the unit vector $k = (v-v_*)/|v-v_*|$.  Thus the Jacobian is 
$$
\left| \frac{d u_i}{ dv_j}  \right| 
= 
\left(1-\frac{\Ctheta}{2}\right)^2\left\{
\left(1-\frac{\Ctheta}{2}\right) + 
\frac{\Ctheta}{2} \ang{k, \sigma}
\right\}.
$$
Since  $b(\ang{k, \sigma}) = 0$ when $\ang{k ,\sigma} \leq 0$ from \eqref{kernelQ}, and  $\Ctheta \in [0,1]$, 
it follows that 
the Jacobian  is bounded from below on the support of the integral \eqref{covterm}. 
But after this change of variables, the old pole $k=(v-v_*)/|v-v_*|$ moves with the angle $\sigma$.  However,  when one takes $\tilde k = (u-v_*)/|u-v_*|$, 
then
$
1- \ang{ k,\sigma } \approx  1 - \left< \right. \! \tilde{k}, \sigma \! \left. \right>, 
$
meaning that the angle to the pole is comparable to the angle to $\tilde{k}$ (which does not vary with $\sigma$).
Thus the estimate analogous to \eqref{bjEST} will continue to hold after this change of variables, giving precisely the estimate in \eqref{covterm}.

It remains to prove \eqref{cancelf21} for $s\in (0, 1/2)$.  This estimate is exactly the same as the one for $\mathcal{D}^{\mathbb{I}}$ except that the cancellation term $\frac{d \CurveEP}{d \Ctheta}(0) \cdot (\tilde{\nabla}   h )(v)$ is unnecessary and we can use \eqref{paradiff1} instead of \eqref{paradiff2} which allows us to take $i =1$ in \eqref{assumeAupper}.
\end{proof}

\begin{proposition}
\label{cancelHprop}
As in Proposition \ref{cancelFprop}, suppose $f$ is a Schwartz function on $\threed$ which is given by the restriction of some Schwartz function in $\mathbb{R}^{n+1}$ to the paraboloid $(v, \frac{1}{2} |v|^2)$ and define $|\tilde{\nabla}|^i f$ analogously for $i=1,2$. 
Then for any $k\ge 0$, we have 
\begin{equation}
\label{cancelh2g1}
 \left| \left( \opGplus   - \opGstar \right)(f,h) \right| 
   \lesssim 
    \Cupper ~ 
2^{(2s-i)k} ~ \nsm |\tilde{\nabla}|^i f\nsm_{L^2_{\gamma+2s}}  
\nsm h\nsm_{L^2_{\gamma+2s}}. 
 \end{equation}
  Here when $s\in (0, 1/2)$ in \eqref{kernelQ} then $i=1$ and when $s \in [1/2, 1)$  we have $i = 2$.
\end{proposition}

\begin{proof}
This proof follows the pattern that is now well-established. 
 The new feature in \eqref{cancelh2g1}   is that, from \eqref{defTKLcarl}, the pointwise differences to examine are
$$
f -  f'
=
\left(  f -  f'  -  \frac{d \CurveEP}{d \Ctheta}(1) \cdot (\tilde{\nabla} f)(v') \right)
 +
 \frac{d \CurveEP}{d \Ctheta}(1) \cdot (\tilde{\nabla} f)(v').
$$
We again 
split 
$
\left( \opGplus   - \opGstar \right)(f,h)
=
\mathcal{D}^{\mathbb{I}}_*+ \mathcal{D}^{\mathbb{II}}_*,
$
where $\mathcal{D}^{\mathbb{I}}_*$ corresponds to the first term in the splitting above.
For the last term $\mathcal{D}^{\mathbb{II}}_*$, we have
$$
\mathcal{D}^{\mathbb{II}}_*
\eqdef
 \int_{\threed} dv'  ~
 \int_{\threed} dv_*   \int_{E_{v_*}^{v'}} d \pi_v  ~  
\tilde{B}_k ~
 g_* ~ h' ~
  \frac{d \CurveEP}{d \Ctheta}(1) \cdot (\tilde{\nabla} f)(v')= 0.
$$
To see this, note that, in this integral, as $v$ varies on circles of constant distance to $v'$, the entire integrand is constant except for $\frac{d \CurveEP}{d \Ctheta}(1)$.  If we write $\frac{d \CurveEP}{d \Ctheta}(1)$ as a sum of two vectors, one lying in the span of the first $n$ directions and the second pointing in the last direction, it follows that we may replace the former vector by its projection onto the direction determined by $v' - v_*$.  But since the original vector points in the direction $v'-v$, the projection vanishes.  Since the last direction of $\frac{d \CurveEP}{d \Ctheta}(1)$  is exactly $\ang{v',v'-v}$, the corresponding integral of this over $v$ also vanishes by symmetry.

To estimate $\mathcal{D}^{\mathbb{I}}_*$ we use \eqref{paradiff1} or \eqref{paradiff2} to observe that
$$
\left| \mathcal{D}^{\mathbb{I}}_* \right|
\lesssim
2^{-ik}
\int_0^1 d \Ctheta
\int_{\threed} dv' \int_{\threed} dv_* \int_{E_{v_*}^{v'}} d \pi_v  ~
\tilde{B}_k ~\ang{v_*}^i \left| g_*  h' \right|~ \left|  |\tilde{\nabla}|^i f(\CurveP(\Ctheta)) \right|. 
$$
In other words, when $s\in (0, 1/2)$, we use \eqref{paradiff1} and obtain $i=1$ and, alternatively, when 
$s\in [1/2,1)$ in \eqref{kernelQ}, we use \eqref{paradiff2} and obtain $i=2$ above.  Since $\mathcal{D}^{\mathbb{II}}_* = 0$ adding the cancellation term $\frac{d \CurveEP}{d \Ctheta}(1) \cdot (\tilde{\nabla} f)(v')$ when $s\in (0, 1/2)$ causes no new problems.
Now we can estimate $\left| \mathcal{D}^{\mathbb{I}}_* \right|$ above using Cauchy-Schwartz, 
\eqref{bjCARLest}, and \eqref{assumeAupper} to conclude that it is uniformly bounded above by a fixed constant times
$\Cupper^{1/2}$ multiplied by
$$
2^{(s-i)k}
\nsm h \nsm_{L^2_{\gamma + 2s}}
\left(
\int_{\threed} dv_*
\ang{v_*}^i \left| g_*  \right|
\int_0^1 d \Ctheta
\int_{\threed} dv'  \int_{E_{v_*}^{v'}} d \pi_v  ~ \tilde{B}_k  \left|  |\tilde{\nabla}|^i f(\CurveP(\Ctheta))\right|^2
\right)^{1/2}. 
$$
The integrals involving 
$|\tilde{\nabla}|^i f(\CurveP(\Ctheta))$ 
are estimated as in \eqref{covterm} after changing from the Carleman representation to the $\sigma$ representation as in \cite[Appendix]{gsNonCutJAMS}.   
\end{proof}

\section{Triple sum estimates for the trilinear form}
\label{sec:upTRI}

The point of this section is to prove Theorem \ref{mainTHM}.  To proceed, we first recall the necessary anisotropic Littlewood-Paley theory adapted to the geometry of the paraboloid, as developed in \cite{gsNonCutJAMS,gsNonCutA}.  After that, we perform the triple sum estimates of the trilinear form using the individual decomposed estimates from Section \ref{physicalDECrel}.

 The generalized Littlewood-Paley projections are given by
\begin{align*}
P_j f(v) & \eqdef \int_{\threed} dv' 2^{\ddim j} \varphi(2^j(\ext{v} - \ext{v'})) \ang{v'} f(v'), \quad j \ge 0,
\\
Q_j f(v) & \eqdef P_j f(v) - P_{j-1} f(v), \quad j \geq 1,
\end{align*}
where $\varphi$ is a $C^\infty$, radial function supported on the unit ball of $\R^\last$ chosen to satisfy various cancellation conditions (we refer to the discussion found in \cite{gsNonCutJAMS} for exactly what is needed) and $\ext{v} \eqdef (v,\frac{1}{2} |v|^2) \in \R^\last$ for any $v \in \threed$.
Informally, $P_j$ corresponds to the  projection onto frequencies at most $2^{j}$ and $Q_j$ corresponds to the  projection onto frequencies comparable to $2^{j}$ (recall that the frequency $2^{j}$ corresponds to the scale $2^{-j}$ in physical space).  We also define $Q_0 \eqdef P_0$. 
These are developed  in 
\cite{gsNonCutJAMS}.
In particular, we  have  $P_j f(v) \rightarrow f(v)$ as $j \rightarrow \infty$ for all sufficiently smooth $f$. 
The principal reason for defining our Littlewood-Paley projections in this way is that the particular choice of paraboloid geometry allows us to control the associated square functions by our anisotropic norm.
Specifically, in our previous papers \cite{gsNonCutJAMS}, it was established that
\begin{equation}
 \sum_{j=0}^\infty 2^{2(s-i)j} \int_{\threed} dv \left| |\tilde{\nabla}|^i Q_j f (v)\right|^2 
 \ang{v}^{\gamma + 2s} 
\lesssim  |f|_{\nspace}^2,
\quad 
(i=0,1,2).
\label{lpsobolev0short}
\end{equation}
This and more general results were proven in \cite[Section 5 \& (5.6)]{gsNonCutJAMS}.  With that, 
we now establish Theorem \ref{mainTHM} via the triple summation estimates in Section \ref{sec:tSum}.

\subsection{The main upper bound inequality}\label{sec:tSum}
We will write 
$
f =   \sum_{j=0}^\infty f_j
$
with the abbreviation
$
Q_j f \eqdef f_j,
$
and likewise for $h$.  Now we expand the trilinear form 
\begin{equation}
\ang{\mathcal{Q}(g, f),h} 
= 
\sum_{l=1}^\infty \sum_{j=0}^\infty   \ang{ f_{j+l}, \mathcal{Q}_g h_j} 
+
\sum_{l=0}^\infty \sum_{j=0}^\infty  \ang{\mathcal{Q}_g^* f_j,h_{j+l}}. 
\label{mainexpand}
\end{equation}
First consider the sum over $l$ of the terms $\ang{ f_{j+l}, \mathcal{Q}_g h_j}$ for fixed $j$.  We expand $\mathcal{Q}$ as a series by introducing the cutoff  terms $\opGplus$ and $\opGminus$ from \eqref{cutoffOP} as follows:
\begin{align}
 \sum_{l=1}^\infty \ang{ f_{j+l}, \mathcal{Q}_g h_j}
 & = 
 \sum_{k=-\infty}^\infty  \sum_{l=1}^\infty 
\left\{ \opGplus(f_{j+l},h_j) - \opGminus(f_{j+l},h_j) \right\} \nonumber \\
& = 
\sum_{k=-\infty}^j \sum_{l=1}^\infty  \left\{ \opGplus(f_{j+l},h_j) - \opGminus(f_{j+l},h_j) \right\}
\label{farsing} \\
& \hspace{30pt} 
+ 
\sum_{l=1}^\infty \sum_{k=j+1}^\infty \left\{ \opGplus(f_{j+l},h_j) - \opGminus(f_{j+l},h_j) \right\}. 
\label{maincancelf}
\end{align}
Now the order of summation may be rearranged at will since the estimates we employ  imply that the sum is absolutely convergent when $h$ and $f$ are Schwartz functions.  Regarding the terms \eqref{farsing}, the inequalities \eqref{tminussmall} and \eqref{tplussmall} dictate that
\[ 
\sum_{k=-\infty}^j  \left|  \opGplus(f_{j+l},h_j) - \opGminus(f_{j+l},h_j)  \right| 
\lesssim \Cupper 2^{2sj}
 \nsm f_{j+l}\nsm_{L^2_{\gamma+2s}} \nsm h_j\nsm_{L^2_{\gamma+2s}}. 
\]
One may now conclude by Cauchy-Schwartz
that
\begin{align*}
 \sum_{j=0}^\infty &
 \sum_{k=-\infty}^j  \! \! \left|  \opGplus(f_{j+l},h_j) - \opGminus(f_{j+l},h_j)  \right| 
 \lesssim 
 \Cupper 
  2^{-sl}
 \sum_{j=0}^\infty 
 2^{s(j+l)} \nsm f_{j+l}\nsm_{L^2_{\gamma+2s}} 
  2^{sj}\nsm h_j\nsm_{L^2_{\gamma+2s}}
 \\
  \lesssim &
  \Cupper
  2^{-sl}  \left| \sum_{j=0}^\infty 2^{2s(j+l)} \nsm f_{j+l}\nsm ^2_{L^2_{\gamma+2s}} \right|^{\frac{1}{2}} \left| \sum_{j=0}^\infty 2^{2sj} \nsm h_j\nsm ^2_{L^2_{\gamma+2s}} \right|^{\frac{1}{2}}
 \lesssim \Cupper 2^{-sl}  \nsm f\nsm_{\nspace} \nsm h\nsm_{\nspace}.
\end{align*} 
The  comparison of the square function norm to the norm $\nsm  \cdot \nsm_{\nspace}$ is provided by \eqref{lpsobolev0short}.  This estimate may clearly  be summed over $l \geq 0$.

To expand $\mathcal{Q}$ for the terms in \eqref{mainexpand} of the form 
$\ang{\mathcal{Q}_g^* f_j,h_{j+l}}$ 
with the operators $\opGplus$, $\opGstar$ and $\opGstarS$ from \eqref{opGlabel} and \eqref{defTKLcarl} we use an analogous argument as follows: 
\begin{align}
\sum_{l=0}^\infty \ang{\mathcal{Q}_g^* f_j,h_{j+l}}& =  \sum_{l=0}^\infty 
\opGstarS(f_j,h_{j+l}) 
+
 \sum_{k=-\infty}^\infty \sum_{l=0}^\infty 
(\opGplus - \opGstar)(f_j,h_{j+l})  
\nonumber \\
& = 
 \sum_{l=0}^\infty 
\opGstarS(f_j,h_{j+l}) 
+ 
\sum_{k=-\infty}^j  \sum_{l=0}^\infty  
(\opGplus - \opGstar)(f_j,h_{j+l}) 
\label{sum1} 
\\
& \hspace{30pt} + 
\sum_{l=0}^\infty \sum_{k=j+1}^\infty 
\left(\opGplus - \opGstar\right)(f_j,h_{j+l})
\label{maincancelh}.
\end{align}
The estimates \eqref{tstarsmall} and \eqref{tplussmall} are 
 used to handle the terms in \eqref{sum1} involving the sum over $k\le j$
just as for  \eqref{farsing}, except that that the roles of $h$ and $f$ are now reversed.
To control the sum over $\opGstarS(f_j,h_{j+l})$  in  \eqref{sum1}, notice that
$$
 \sum_{l=0}^\infty \sum_{j=0}^\infty 
\left| \opGstarS(f_j,h_{j+l}) \right|
 \lesssim  \Cupper
  \sum_{l=0}^\infty  2^{-sl} \sum_{j=0}^\infty ~ 2^{sj}\nsm f_j\nsm_{L^2_{\gamma}}  
2^{s(j+l)}\nsm h_{j+l}\nsm_{L^2_{\gamma}},
$$
which follows from Proposition \ref{opGstarEST}
and the inequality
$
1 \le 2^{2sj} = 2^{-sl}  2^{s(j+l)} 2^{sj}.
$
Again the Cauchy-Schwartz inequality  and \eqref{lpsobolev0short} yield the desired upper bound.

Lastly, consider 
\eqref{maincancelf}
and
\eqref{maincancelh}. 
The  terms \eqref{maincancelf} are handled by \eqref{cancelf21}; we have 
$$
 \left| \left( \opGplus  - \opGminus \right)(f_{j+l},h_j) \right|
\lesssim 
\Cupper
2^{(2s-i)k} ~ \nsm  f_{j+l}\nsm_{L^2_{\gamma+2s}}  
\nsm |\tilde{\nabla}|^i h_j\nsm_{L^2_{\gamma+2s}}. 
$$
Now there is decay of the norm as $k \rightarrow \infty$ since $2s - i < 0$, so that
$$
 \sum_{k=j+1}^\infty   \left| \left( \opGplus  - \opGminus \right)(f_{j+l},h_j) \right|
 \lesssim 
 \Cupper
 2^{(2s-i)j} 
 \nsm  f_{j+l}\nsm_{L^2_{\gamma+2s}}  
\nsm |\tilde{\nabla}|^i h_j\nsm_{L^2_{\gamma+2s}}. 
$$
Again Cauchy-Schwartz is applied to the sum over $j$.  In this case $2^{(2s-i)j}$ is written as $2^{(s-i)j} 2^{s(j+l)} 2^{-sl}$; the first factor goes with $h$, the second with $f$, and the third remains for the sum over $l$.  Once again \eqref{lpsobolev0short} is employed.  
The desired bound for the trilinear term is completed by performing summation of the terms \eqref{maincancelh}.  The pattern of inequalities is the same, this time using \eqref{cancelh2g1}.  In particular, one has 
$$
 \left| \left(\opGplus - \opGstar\right)(f_j,h_{j+l}) \right|
\lesssim 
 \Cupper
2^{(2s-i)k}  \nsm |\tilde{\nabla}|^i  f_{j}\nsm_{L^2_{\gamma+2s}} \nsm h_{j+l} \nsm_{L^2_{\gamma+2s}}.
$$
Again, as in Proposition \ref{cancelHprop}, $i=1,2$ always satisfies $2s - i < 0$, leading to 
$$
 \sum_{k=j+1}^\infty 
  \left| \left(\opGplus - \opGstar\right)(f_j,h_{j+l}) \right|
 \lesssim 
  \Cupper
2^{(2s-i)j}  \nsm |\tilde{\nabla}|^i  f_{j}\nsm_{L^2_{\gamma+2s}} \nsm h_{j+l} \nsm_{L^2_{\gamma+2s}}.
$$
The same Cauchy-Schwartz estimate is used for the sum over $j$; there  is exponential decay allowing the sum over $l$ to be estimated.  
The end result is Theorem \ref{mainTHM}.

\section{The coercive lower bound inequality}  
\label{sec:mainCOER}

The goal of this section is to prove \eqref{mainlower} in Theorem \ref{mainLOWERthm}.  The main idea of the proof can be stated as follows:  the principal analytical difference between the left- and right-hand sides of \eqref{mainlower} is that (for fixed $v_*$) the variables $v$ and $v'$ are constrained relative to one another on the left-hand side, while on the right-hand side they essentially are not (they need only satisfy a distance inequality). To ``regularize'' the left-hand side and remove the constraint, we will exploit the fact that, for different values of $v_*$, the constraint (namely, the sphere) between $v'$ and $v$ changes. 
 We will employ an elementary but completely novel convolution-type argument which will allow us to exploit the changing constraint.  Loosely speaking, because the spheres are not static, the convolution-type operation will ``smear out'' the support of the integral and give something analogous to the right-hand side.

\subsection{Regularization}\label{sec:regul}

As is customary, the proof proceeds by a careful dyadic decomposition.
For any integer $k$, consider the set $\Omega_k=\Omega_k(v,v',v_*)$ given by
\[ 
\Omega_k \eqdef \set{ (v,v',v_*)}{ |v-v'| \leq 2^{-k} \mbox{ and } \ang{2 v' - v - v_*, v - v_*} \geq 0}.
\]
Furthermore the condition 
$\ang{2 v' - v - v_*, v - v_*} \geq 0$ is needed because of the support condition in \eqref{kernelQ}.  
We also use
 the quadratic functional given by
\begin{align*} 
I_k(f) & \eqdef \int_{\ballR} dv_* \int_{\threed} dv \int_{\dsp_{v_*}^v} d \sigma_{v'} ~ (f' - f)^2 ~ {\mathbf 1}_{\Omega_k} ~ g_*
~ |v-v_*|^{(n-1)+ \gamma + 2s}. 
\end{align*}
By the integral $d \sigma_{v'}$, we mean
$
\int_{\sph} d\sigma ~ \phi(v') = \int_{\dsp_{v_*}^v} d \sigma_{v'}  \phi(v'),
$
where the left-hand side is exactly as in \eqref{sigma}. 
In particular, note that $\dsp_{v_*}^{v}$ is the sphere
\begin{equation}
\dsp_{v_*}^{v} \eqdef \set{ w \in \threed }{ 0 = \ang{w-v, w - v_*} },
\label{sphereUN}
\end{equation}
which is the unique sphere for which $v$ and $v_*$ are antipodal, having center $\frac{v+v_*}{2}$ and radius $\frac{|v-v_*|}{2}$.
With \eqref{sinEXPest} and \eqref{cancellationSPLIT}
it is easy to check that
\begin{equation} 
N_g(f)
 \gtrsim \sum_{k=-\infty}^\infty 2^{k((n-1) + 2s)} I_k (f).
\label{mainlower0}
\end{equation}
This follows because
\[
b(\cos\theta) \gtrsim \left( \frac{|v-v_*|}{|v-v'|} \right)^{n-1+2s}
\]
by virtue of \eqref{kernelQ} and \eqref{sigma}; with this inequality, \eqref{mainlower0} follows from the observation
\[ 
\sum_{k=-\infty}^{\infty} 2^{k((n-1) + 2s)} ~ {\mathbf 1}_{\Omega_k} \lesssim |v-v'|^{-(n-1)-2s}, 
\]
since the sum on the left-hand side is a geometric series which terminates at some maximal $k$ satisfying $2^{k_{max}} \approx |v-v'|^{-1}$ (and the sum is comparable to the value of the largest term).
The goal will be to estimate the terms on the right side of \eqref{mainlower0} by something more directly comparable to our semi-norm \eqref{normDOTdef}.  To that end, let
\begin{align*}
w_k(v,v_*) \eqdef \int_{\dsp_{v_*}^v} d \sigma_{v'} {\mathbf 1}_{\Omega_k}  \lesssim 2^{-k(n-1)} |v-v_*|^{-(n-1)},
\end{align*}
(the estimate from above follows because the angle to the pole is comparable to $|v - v'| |v-v_*|^{-1}$). 
Extending our convention by defining
$
 \overline{g}_*
 \eqdef
 g( \overline{v}_*)
$
and
$
\overline{f}'
\eqdef
f(\overline{v}'),
$ 
Fubini's theorem guarantees that the quantity $I_k(f) 
\int_{\ballR} dv_* g_*$ is equal to both of the following integrals:
\begin{align*}
 & \int_{\threed} \! \! dv \!  \int_{\ballR} \! \! d \overline{v}_* \! \int_{\dsp_{\overline{v}_*}^v} \!  \! \! d \sigma_{\overline{v}'} \! \int_{\ballR} \!  \! dv_* \! \int_{\dsp_{v_*}^v} \!  \! \! d \sigma_{v'} (f' - f)^2 
 \frac{ {\mathbf 1}_{\Omega_k}  {\mathbf 1}_{\overline{\Omega}_k} }{w_k(v,\overline{v}_*)} ~ g_* \overline{g}_*  ~
 |v-v_*|^{(n-1)+\gamma + 2s}, 
 \\
 & \int_{\threed} \!  \! dv \! \int_{\ballR} \!  \! d \overline{v}_* \! \int_{\dsp_{\overline{v}_*}^v} \!  \! \! d \sigma_{\overline{v}'} \! \int_{\ballR} \!  \! dv_* \! \int_{\dsp_{v_*}^v} \!  \! \!  d \sigma_{v'} (\overline{f}' - f)^2 
 \frac{{\mathbf 1}_{\Omega_k}  {\mathbf 1}_{\overline{\Omega}_k}}{w_k(v,v_*)} ~ g_* \overline{g}_* ~
 |v - \overline{v}_*|^{(n-1)+\gamma+2s},
\end{align*}
where
$\overline{\Omega}_k=\Omega_k(v,\overline{v}',\overline{v}_*)$.
We now bound $I_k(f) \int_{\ballR} dv_* g_*$ below by an integral whose integrand is the average of the two integrands above.  The elementary inequality 
\begin{equation}
     (f' - f)^2
 + 
 (\overline{f}' - f)^2
\geq \frac{1}{2}
   ( \overline{f}' - f')^2 ,
   \label{ineqUPPERfT}
\end{equation}
leads to the somewhat less elementary observation that 
\begin{multline}
 ( f' - f)^2  
 \frac{|v-v_*|^{(n-1)+\gamma+2s}}{w_k(v,\overline{v}_*)} 
 + 
 (\overline{f}' - f)^2 \frac{|v-\overline{v}_*|^{(n-1)+\gamma+2s} }{w_k(v,v_*)} 
 \\
 \gtrsim 2^{k(n-1)}   \frac{\min \{|v-v_*|^{\gamma+2s}, |v-\overline{v}_*|^{\gamma+2s}\}}{(|v-v_*||v-\overline{v}_*|)^{-(n-1)}}   ( \overline{f}' - f')^2, 
\label{quadmin}
\end{multline}
where we have also used the estimate of $w_k$ from above.  

Since the quadratic difference now involves only $v'$ and $\overline{v}'$, the next  step is to perform a  change of variables {\it simultaneously} for both $v'$ and $\overline{v}'$ (motived by the Carleman-type representations) so that, when we integrate the right-hand side of \eqref{quadmin} over all the relevant variables, we may treat $v'$ and $\overline{v}'$ as being unconstrained at the price of requiring $v$ to simultaneously satisfy {\it two} constraints.
Let $H$ be any (Borel) measurable function of $v, v'$ and $\overline{v}'$.  Then
for any fixed $v_*$ and $\overline{v}_*$, with the notation $\dsp_{v_*}^{v}$
from \eqref{sphereUN} we have that
\begin{multline} \label{doubleCARL}
 \int_{\threed} \! dv   \int_{\dsp_{\overline{v}_*}^v} \! d \sigma_{\overline{v}'} \int_{\dsp_{v_*}^v} \! d \sigma_{v'}
 ~  H(v,v',\overline{v}') 
 \\
 = \int_{\threed} \! dv' \int_{\threed} \! d\overline{v}' \int_{E_{v',v_*}^{\overline{v}',\overline{v}_*}} \! d \pi_{v} \frac{H(v,v',\overline{v}')}{|v-v_*|^{n-2} |v - \overline{v}_*|^{n-2} D^{\frac{1}{2}}}. 
\end{multline}
where 
\begin{equation}
D \eqdef |v' - v_*|^2 |\overline{v}' - \overline{v}_*|^2 - \ang{ v' - v_*, \overline{v}' - \overline{v}_*}^2,
 \label{quantityDdef}
\end{equation}
and $E_{v',v_*}^{\overline{v}',\overline{v}_*}$ is the co-plane (by which we mean an affine subspace of codimension 2) of points $v \in \threed$ satisfying the constraints
%
\begin{equation}
E_{v',v_*}^{\overline{v}',\overline{v}_*}
\eqdef 
\{v \in \threed : ~
 0  = \ang{v-v', v' - v_*}, ~
 0  = \ang{v-\overline{v}',  \overline{v}' - \overline{v}_*}
 \},
 \label{coPlaneEdef}
\end{equation}
and, as always, $d \pi_v$ is the Lebesgue measure on this co-plane.
Now consider the quantity $K_k=K_k(v',v_*, \overline{v}',\overline{v}_*)$ given by 
\begin{equation} 
K_k \eqdef   \frac{2^{k(n-1)}}{D^\frac{1}{2}} 
\int_{E_{v',v_*}^{\overline{v}',\overline{v}_*}} d \pi_v ~ \frac{ \min \{ |v -v_*|^{\gamma+2s}, |v - \overline{v}_*|^{\gamma+2s} \} }{|v-v_*|^{-1} |v-\overline{v}_*|^{-1}} {\mathbf 1}_{\Omega_k}{\mathbf 1}_{\overline{\Omega}_k}.
\label{newkern}
\end{equation}
By the co-plane change of variables \eqref{doubleCARL}, we may conclude that
\begin{equation} I_k(f) \int_{\ballR} dv_* g_* \gtrsim    
 \int_{\ballR}  dv_* ~ g_* 
 \int_{\ballR}  d \overline{v}_* ~ \overline{g}_*
 \int_{\threed} \!  \! d \overline{v}' \!  \! \int_{\threed} \!  \! d v' K_k (f' - \overline{f}')^2.
\label{itok}
\end{equation}
The goal at this point is, of course, to obtain a favorable estimate for $K_k$ in terms of the geometry of the various variables of integration.

\subsection{Estimation of \eqref{newkern}}
Regarding the function $D$ from \eqref{quantityDdef}, it is a useful fact to note that $D$ is a Gram determinant, 
and its value is equal to $4$ times the square of the area of the triangle with vertices $0$, $v' -v_*$ and $\overline{v}' - \overline{v}_*$.  Consequently we have two alternate formulas for $D$ which will also be useful:
\begin{align*}
D 
   & = |v' - v_*|^2 |v' - \overline{v}' + \overline{v}_* - v_*|^2 - 
   \ang{v' - \overline{v}' + \overline{v}_* - v_* , v' - v_*}^2 
   \\
 & = |\overline{v}' - \overline{v}_*|^2 |v' - \overline{v}' + \overline{v}_* - v_*|^2 - 
 \ang{v' - \overline{v}' + \overline{v}_* - v_* , \overline{v}' - \overline{v}_*}^2.
\end{align*}
For the moment, let us make a series of assumptions (to be examined later).  To that end, we will consider 
$v'$, $v_*$, $\overline{v}'$, $\overline{v}_*$ as fixed and satisfying the inequalities
\begin{eqnarray}
\label{asump1} 
 |\overline{v}_* - v_*| & \geq &\delta, 
\\
D  &\geq & \delta^2 \left(|v' - v_*|^2 + |\overline{v}' - \overline{v}_*|^2\right), 
\label{asump3}
\end{eqnarray}
for some fixed constant $\delta>0$.   Now \eqref{mainlower} in Theorem \ref{mainLOWERthm} will be proved through a series of propositions estimating the various features of \eqref{newkern}.

\begin{proposition}
Assume that \eqref{asump1} and \eqref{asump3} hold for some quadruple $(v',v_*, \overline{v}',\overline{v}_*)$ with fixed $\delta$.  Then for all $k\ge k_0$, with $k_0 = k_0(n, R, \delta)\ge 0$, we have
\begin{equation}
 K_k \gtrsim 2^{k(n-1)} \ang{v'}^{\gamma + 2s + 1} \int_{E_{v',v_*}^{\overline{v}',\overline{v}_*}} d \pi_v 
 ~ {\mathbf 1}_{\Omega_k} ~ {\mathbf 1}_{\overline{\Omega}_k}. \label{estlengths}
\end{equation}
\end{proposition}

\begin{proof}
The inequality \eqref{estlengths} rests on a variety of length inequalities which are a consequence of the assumptions \eqref{asump1} and \eqref{asump3}.  First, assuming $k \geq 0$, on $\overline{\Omega}_k \cap \Omega_k$  we have that $|v' - \overline{v}'| \leq 2$.  By the lower bound for $D$, this implies that
\[ 
\delta^2 (|v' - v_*|^2 + |\overline{v}' - \overline{v}_*|^2)  \leq |\overline{v}' - \overline{v}_*|^2 |v' - \overline{v}' + \overline{v}_* - v_*|^2 \leq (2R+2)^2 |\overline{v}' - \overline{v}_*|^2,
\]
(where the bound from above for $D$ is obtained by dropping the negative inner product squared in the third representation of $D$ mentioned previously)
so that $|v' - v_*| \leq \delta^{-1} (2R + 2) |\overline{v}' - \overline{v}_*|$.  By symmetry, $ |\overline{v}' - \overline{v}_*| \leq \delta^{-1} (2R + 2)|v' - v_*|$ as well.  As long as $k$ is large enough that $2^{-k} \leq \frac{\delta}{4}$, then, we also have that
\[ 
|\overline{v}' - \overline{v}_*| + |v' - v_*| \geq |v_*  - \overline{v}_*| - |v' - \overline{v}'| \geq  \frac{\delta}{2}. 
\]
In particular, we may conclude that $|\overline{v}' - \overline{v}_*| \approx |v' - v_*| \gtrsim 1$ with constants depending only on $\delta$ and $R$.
So again, if $k$ is large enough (depending on $\delta$ and $R$) this means that any $v$ in 
$\Omega_k \cap \overline{\Omega}_k \cap E_{v',v_*}^{\overline{v}',\overline{v}_*}$ will satisfy $|v - v_*| \approx |v' - v_*| \approx |v-\overline{v}_*| \approx \ang{v'}$ (because in this case all of the lengths are bounded from below on 
\eqref{coPlaneEdef}, $|v-v'|$ is bounded from above, and $|v_*|, |\overline{v}_*|$ are bounded above as well).  Thus
\[ 
K_k \gtrsim \frac{2^{k(n-1)}}{D^{\frac{1}{2}}} \ang{v'}^{\gamma + 2s + 2} \int_{E_{v',v_*}^{\overline{v}',\overline{v}_*}} d \pi_v  ~{\mathbf 1}_{\Omega_k}~{\mathbf 1}_{\overline{\Omega}_k}, 
\]
uniformly (where the constant depends on the dimension, $R$, and $\delta$ only).  Finally, using the bound $D \leq (2R+2)^2 |v' - v_*|^2$
 (obtained by keeping only the product of squared lengths and neglecting the squared inner product) gives \eqref{estlengths}.
\end{proof}

\begin{lemma}
\label{ballprop} Under the assumptions \eqref{asump1} and \eqref{asump3} for some fixed quadruple of points $(v',v_*, \overline{v}',\overline{v}_*)$, there is a constant $\varepsilon>0$ and a $k_0$ such that the additional constraint $d(v',\overline{v}') \leq 2^{-k} \varepsilon$ for any $k \geq k_0$ implies that the set
\[ 
\set{v \in E_{v',v_*}^{\overline{v}',\overline{v}_*} }{ |v-v'| \leq 2^{-k} \mbox{ and } |v - \overline{v}'| \leq 2^{-k}}, 
\]
is nonempty and contains a {\it Euclidean} ball of radius $\frac{3}{4} 2^{-k}$.
\end{lemma}

\begin{proof}
The objective at this point is to show that, when $v'$ and $\overline{v}'$ are close in the anisotropic sense, then the co-plane above passes near to these points in the {\it isotropic} sense.  To that end, consider the temporary definitions
\[ 
u \eqdef v' - v_*,
\quad
 \mbox{ and }
\quad
 \overline{u} \eqdef \overline{v}'- \overline{v}_*. 
\]
Note that \eqref{quantityDdef} is given simply by the expression
$
D = |u|^2 |\overline{u}|^2 - \ang{u,\overline{u}}^2. 
$
It is a straightforward exercise in linear algebra 
to verify that the vector $w$ given by
\begin{equation}
w \eqdef v' + \frac{\ang{\overline{v}'-v',\overline{u}}}{D} \left( - \ang{u, \overline{u}} u + |u|^2 \overline{u} \right), 
\label{wDEFmin}
\end{equation}
must lie in the co-plane $E_{v',v_*}^{\overline{v}',\overline{v}_*}$ from \eqref{coPlaneEdef} (in fact, $w$ is the point in the co-plane of minimal Euclidean distance to $v'$).  The distance of this point $w$ to $v'$ is given by
\begin{multline*}
 |w-v'|^2  = \left| \frac{\ang{\overline{v}'-v',\overline{u}}}{D} \right|^2 
 \left( \ang{u,\overline{u}}^2 |u|^2 - 2 \ang{u,\overline{u}} |u|^2 \ang{u,\overline{u}}  + |u|^2 |u|^2 |\overline{u}|^2\right) 
 \\
  = \frac{|u|^2 \ang{\overline{v}'-v',\overline{u}}^2}{D}.
\end{multline*}
Expanding this inner product gives
\[ \ang{\overline{v}' - v', \overline{v}' - \overline{v}_*} = \frac{1}{2} (|v'|^2 - |\overline{v}'|^2 - |\overline{v}'-v'|^2) - \ang{\overline{v}' - v', \overline{v}_*}; \]
consequently, if $d(v', \overline{v}') \leq 1$ then it follows that
\[ 
|w - v'|^2 \lesssim \frac{|u|^2 \ang{\overline{v}_*}^2 (d (v', \overline{v}'))^2}{D} \lesssim (d(v',\overline{v}'))^2, 
\]
where the final constant depends on $\delta$ and $R$.

The crucial point is that if $d(v',\overline{v}') \leq \varepsilon 2^{-k}$ for  $\varepsilon=\varepsilon(n,\delta,R)>0$ sufficiently small, then $w$ will be within $\frac{1}{8} 2^{-k}$ of $v'$: $|w-v'| \leq \frac{1}{8} 2^{-k}$.  We may also choose $\varepsilon$ so that $|v' - \overline{v}'| \leq \frac{1}{8} 2^{-k}$, whence the triangle inequality guarantees $|\overline{v}' - w| \leq \frac{1}{4} 2^{-k}$.  Another application of the triangle inequality shows that any $v \in E_{v',v_*}^{\overline{v}',\overline{v}_*}$  satisfying $|v - w| < \frac{3}{4} 2^{-k}$ will be forced to satisfy both $|v - v'| < 2^{-k}$ and $|v - \overline{v}'| < 2^{-k}$.
\end{proof}

\begin{proposition} \label{finalk}
Under the same hypotheses as Lemma \ref{ballprop}, we have the lower bound
\begin{equation} K_k \gtrsim 2^{k} \ang{v'}^{\gamma+2s+1} {\mathbf 1}_{d(v',\overline{v}') \leq \varepsilon 2^{-k}}, 
\notag
\end{equation}
uniformly for all $k\ge k_0$, where $k_0 = k_0(n,\delta,R)>0$ and  $\varepsilon = \varepsilon(n,\delta,R)>0$.
\end{proposition}

\begin{proof}
When $d(v', \overline{v}') \leq \varepsilon 2^{-k}$, then Lemma \ref{ballprop} identifies a Euclidean ball of radius comparable to $2^{-k}$ which lies at a distance strictly less than $2^{-k}$ to both the points $v'$ and $\overline{v}'$.  This is nearly sufficient to assert that this ball lies in the intersection $\Omega_k \cap \overline{\Omega}_k$, but there is a second constraint to satisfy: namely, the positivity of the relevant inner products.  For any $v$ in this ball, on 
 \eqref{coPlaneEdef}
 we have 
\begin{align*} 
\ang{2v' - v - v_*, v - v_*} & = \ang{v' - v_* - (v - v'), v' - v_* + (v - v')} \\
& = |v' - v_*|^2 - |v - v'|^2.
\end{align*}
Again as long as $k$ is sufficiently large, this quantity will be positive 
because $|v' - v_*|^2$ is bounded below.  By symmetry, $\ang{2\overline{v}' - v - \overline{v}_*, v - \overline{v}_*}$ will be positive in this case as well.
Consequently, for any fixed $v',\overline{v}'$ with $d(v',\overline{v}') \leq \varepsilon 2^{-k}$, the ball of points $v \in E_{v',v_*}^{\overline{v}',\overline{v}_*}$ identified in Lemma \ref{ballprop} 
will be contained in $\Omega_k \cap \overline{\Omega}_k$:
\[ 
{\mathbf 1}_{\Omega_k}{\mathbf 1}_{\overline{\Omega}_k}
=
{\mathbf 1}_{\Omega_k}(v,v',v_*) {\mathbf 1}_{\Omega_k}(v,\overline{v}',\overline{v}_*) \geq {\mathbf 1}_{|v - w| < \frac{3}{4} 2^{-k}}. 
\]
Here $w$ is defined in \eqref{wDEFmin}.
Since the measure of this set in $E_{v',v_*}^{\overline{v}',\overline{v}_*}$ is comparable to $2^{-(n-2)k}$ (it is codimension $2$ in $\threed$), we have arrived at the conclusion
\[ 
K_k \gtrsim 2^{k(n-1)} \ang{v'}^{\gamma+2s+1} 2^{-k(n-2)} {\mathbf 1}_{d(v',\overline{v}') \leq \varepsilon 2^{-k}}, 
\]
uniformly for all $k\ge k_0$, which is exactly the desired lower bound for \eqref{newkern}.  
\end{proof}

In light of \eqref{itok} and Proposition \ref{finalk} we may now establish a temporary version of \eqref{mainlower}.
In particular, by \eqref{itok} and Proposition \ref{finalk} 
we have that, for all $k \geq k_0$, $I_k(f) |g|_{L^1(B_R)}$ is bounded uniformly below by
\begin{equation}
 2^k \int_{\threed}   d \overline{v}' \int_{\threed}  d v'  ~ (f' - \overline{f}')^2
 ~ \ang{v'}^{\gamma+2s+1}  {\mathbf 1}_{d(v',\overline{v}') \leq \varepsilon 2^{-k}}
~\AGthing (v',\overline{v}'), \label{mainbelow}
\end{equation}
with
\[
  \AGthing
  \eqdef
\int_{\ballR}  d \overline{v}_* \int_{\ballR}  dv_*
 ~ g_* \overline{g}_* {\mathbf 1}_G.
\]
By \eqref{mainlower0}, one may estimate $N_g(f)~  \nsm g \nsm_{L^1(\ballR)}$ from below by multiplying \eqref{mainbelow} by $2^{((n-1) + 2s)k}$ and summing over $k$.
Interchanging summation and integration of the terms \eqref{mainbelow}, we will next employ the inequality
\begin{equation}
\sum_{k=k_0}^\infty  2^{k(n + 2s)}   {\mathbf 1}_{d(v', \overline{v}') \leq \varepsilon 2^{-k}}  
  \gtrsim  d(v',\overline{v}')^{-n-2s} {\mathbf 1}_{d(v',\overline{v}') \leq 2^{-k_0} \varepsilon},
\label{lowersumineq}
\end{equation}
which follows because, as before, the sum on the left is a {\it finite} geometric series (for any pair $v', \overline{v}'$, the characteristic functions are zero if $2^{-k} \varepsilon > d(v', \overline{v}')$).  The sum of a (nontrivial) geometric series is comparable to its largest term, which, in this case, satisfies $2^{-k_{max}} \approx d(v',\overline{v}')$.  
Thus by \eqref{lowersumineq} it follows that
\begin{equation}
N_g(f)~  \nsm g \nsm_{L^1(\ballR)}
  \gtrsim 
  \int_{\threed} \!  \! d \overline{v}'  \! \int_{\threed} \!  \! d v' \frac{(f' - \overline{f}')^2}{d(v',\overline{v}')^{n+2s}} \ang{v'}^{\gamma+2s+1}  {\mathbf 1}_{d(v',\overline{v}') \leq 2^{-k_0}\varepsilon}
 ~ \AGthing.
\label{lowerMAINsum}
\end{equation}
The proof of \eqref{mainlower} would be complete if we had a uniform positive lower bound for  $\AGthing(v',\overline{v}')$, and if we could  replace $d(v',\overline{v}') \leq 2^{-k_0} \varepsilon$ with  $d(v',\overline{v}') \leq 1$ in the above.  Closing both of these issues is the content of the remainder of this section.

\subsection{Simplifying the geometry}
The next task at hand is to gain a better geometric understanding of  \eqref{asump3}. We seek to replace this somewhat obscure condition with a more intuitive one.  Specifically, we will replace \eqref{asump1} and \eqref{asump3} with:
\begin{eqnarray}
\label{asump4}
|v' - v_*| & \geq & \delta>0, 
\\
\label{asump6} 
|v' - v_*|^2 |v' - \overline{v}_*|^2 - \ang{v' - v_*, v' - \overline{v}_*}^2  &\geq & \delta^2 |v' - v_*|^2.
\end{eqnarray}
These new conditions have simple geometric interpretations: the distance from $v'$ to $v_*$ is at least $\delta$, and $\overline{v}_*$ lies outside a tube of radius $\delta$ around the line through $v'$ and $v_*$.

Note also that \eqref{asump6} itself implies $|v' - \overline{v}_*| \geq \delta$ as well (which may be seen by neglecting the squared inner product on the left-hand side). 
Likewise, the equality
\[ 
|v' - v_*|^2 |v' - \overline{v}_*|^2 - \ang{v' - v_*, v' - \overline{v}_*}^2 
= 
|v_* - \overline{v}_*|^2 |v' - v_*|^2 - \ang{v_* - \overline{v}_*,v' - v_*}^2, 
\]
guarantees that $|v_* - \overline{v}_*| \geq \delta$ as well.  Now since $|v' - \overline{v}_*|$ and $|v' - v_*|$ are bounded below and the distances $|v_* - \overline{v}_*|$ and $|v' - \overline{v}'|$ are bounded above, it follows that 
$$
|v' - \overline{v}_*| \approx |\overline{v}' - \overline{v}_*| \approx |v' - v_*| \approx |\overline{v}' - v_*| \approx \ang{v'},
$$ 
when $d(v',\overline{v}') \ll 1$; all of the implicit constants will  depend on $\delta$ and $R$.

\begin{proposition}\label{changeasump}
The inequalities \eqref{asump4} and \eqref{asump6} imply assumptions \eqref{asump1} and \eqref{asump3} 
(with a different $\delta>0$) provided $d(v',\overline{v}') \leq \varepsilon$ 
for some small $\varepsilon=\varepsilon(\delta,R)$. 
\end{proposition}

\begin{proof}
It remains only to verify \eqref{asump3}.  Note that \eqref{asump6} coincides with \eqref{asump3} when $\overline{v}' = v'$.  It thus suffices to estimate the difference when going from $v'$ to $\overline{v}'$ in the appropriate places.
First, the change in the squared length $|v' - \overline{v}' + \overline{v}_* - v_*|^2$ is
\begin{align*}
 \left| |v' - \overline{v}' + \overline{v}_* - v_*|^2 - |\overline{v}_* - v_*|^2 \right| & \leq  |v' - \overline{v}'|^2 + 2 |v' - \overline{v}'|| \overline{v}_* - v_*| \\
& \lesssim |v' - \overline{v}'|,
\end{align*}
where we use the inequalities $|v'-\overline{v}'| \leq 1$ and $|v_* - \overline{v}_*| \leq 2R$.  
Next, we observe the equality
\begin{align*}
 \ang{v' - \overline{v}', v' - v_*}  & =  \ang{v' - \overline{v}',v'} - \ang{v' - \overline{v}',v_*}   
 \\
&  =  \frac{1}{2} ( |v'|^2 - |\overline{v}'|^2 + |v' - \overline{v}'|^2 ) - \ang{v' - \overline{v}',v_*},
\end{align*}
meaning that $|\ang{v' - \overline{v}', v' - v_*}| \lesssim d(v',\overline{v}')$.  In particular, this implies
\begin{align*}
\left| \vphantom{1^2} \right.  \left< \right. v' - \overline{v}' + & \overline{v}_* - v_*,  v' - v_* \left. \right>^2 - \left. \ang{\overline{v}_* - v_*, v' - v_* }^2 \right| \\ 
& = \left| \ang{v' - \overline{v}', v' - v_*}  \right| \left| \ang{v' - \overline{v}' + \overline{v}_* - v_*, v' - v_*} + \ang{\overline{v}_* - v_*, v' - v_* } \right| \\
& \lesssim |v' - v_*| d(v',\overline{v}').
\end{align*} 
Since $|v' - v_*| \geq \delta$, it follows that the difference of the left-hand sides of \eqref{asump3} and \eqref{asump6} are bounded in magnitude by a uniform constant times $d(v', \overline{v}') |v' - v_*|^2$ (because $|v' - v_*|$ is bounded below, the squared length dominates the length itself).  Assuming that $d(v', \overline{v}')\ll 1$  and using  $|v' - v_*| \approx |\overline{v}' - \overline{v}_*|$ establishes \eqref{asump3}.
\end{proof}

\begin{proposition}\label{propGlower}
Suppose that \eqref{assumeAlower} holds for some $\Clower>0$.  Then
\begin{equation} 
\AGthing =
\int_{\ballR} dv_* \int_{\ballR} d \overline{v}_*~ {\mathbf 1}_{G} ~g_* ~\overline{g}_*
\ge \Clower^2.
\notag
\end{equation}
\end{proposition}

\begin{proof}
By  Proposition \ref{changeasump}, it suffices to estimate the integral $\AGthing$
uniformly from below for all $v'$ with the set $G$ replaced by the set $G'$ determined by the constraints \eqref{asump4} and \eqref{asump6}.  Clearly for each fixed $v_*$ with $|v_* - v'| \geq \delta$ the integral over $\overline{v}_*$ is exactly the integral of $g$ over the ball $\ballR$ minus a tube of radius $\delta$ determined by $v'$ and $v_*$.  Thus the Fubini theorem dictates that
\[ 
\int_{\ballR} dv_* ~\int_{\ballR} d \overline{v}_* ~{\mathbf 1}_{G'} ~g_* \overline{g}_* 
\geq \Clower \int_{\ballR \setminus B_\delta(v')} dv_* ~ g_*. 
\]
To conclude, we observe that $\ballR$ minus the ball of radius $\delta$ around $v'$ is a strictly larger set than the ball minus {\it any} tube of radius $\delta$ whose axis passes through $v'$.
\end{proof}

At this point, with Proposition \ref{propGlower}, we may estimate $\AGthing$ and consequently use \eqref{lowerMAINsum} to deduce \eqref{mainlower} aside from the limitation that the distance $d(v',\overline{v}')$ would be constrained to be less than or equal to $\varepsilon>0$ rather than the larger bound $1$ appearing in the definition of $N^{s,\gamma}$.  This turns out to be a minor difference which is fixed in the next subsection.

\subsection{The completion of \eqref{mainlower}}
For completeness, we eliminate the parameter $2^{-k_0} \varepsilon$ appearing in 
\eqref{lowerMAINsum}
and
\eqref{lowersumineq}.  
We will use the estimate \eqref{ineqUPPERfT} (with the roles of $f$ and $\overline{f}'$ reversed) combined with  the trick of adding extra variables of integration and bounding below.  First we add extra variables of integration:
\begin{align*}
J_{\varepsilon}(f)
\eqdef
 \int_{\R^n} dv' & \int_{\R^n} d \overline{v}' (f' - \overline{f}')^2 \ang{v'}^{\gamma+2s+1} {\mathbf 1}_{d(v',\overline{v}') \leq \varepsilon} \\
 & = \int_{\R^n} dv \int_{\R^n} dv' \int_{\R^n} d \overline{v}' (f' - \overline{f}')^2 \ang{v'}^{\gamma+2s+1} {\mathbf 1}_{d(v',\overline{v}') \leq \varepsilon} \frac{{\mathbf 1}_{d(v,\overline{v}') \leq \varepsilon}}{|\tilde B_\varepsilon(\overline{v}')|}
\\
& = \int_{\R^n} dv \int_{\R^n} dv' \int_{\R^n} d \overline{v}' (f - \overline{f}')^2 \ang{v}^{\gamma+2s+1} \frac{{\mathbf 1}_{d(v',\overline{v}') \leq \varepsilon}}{|\tilde B_\varepsilon(\overline{v}')|} {\mathbf 1}_{d(v,\overline{v}') \leq \varepsilon}.
\end{align*}
Here $\tilde B_\varepsilon(\overline{v}')$ is the anisotropic ball of radius $\varepsilon$ 
centered at $\overline{v}'$ and $| \tilde B_\varepsilon(\overline{v}')|$ is its measure.  Now we use \eqref{ineqUPPERfT} (together with the properties $\ang{v} \approx \ang{v'}$ and $|\tilde B_\varepsilon (\overline{v}')| \approx |\tilde B_\varepsilon (v')|$, which hold uniformly when $\varepsilon \leq 1$) to conclude that
$$
J_{\varepsilon}(f)+J_{\varepsilon}(f)
  \gtrsim \int_{\threed} dv \int_{\threed} dv' (f - f')^2 \ang{v'}^{\gamma+2s+1} H(v,v'),
$$
where
\[ 
H(v,v') \eqdef \frac{1}{|\tilde B_\varepsilon(v')|} \int_{\threed} d \overline{v}' ~ {\mathbf 1}_{d(v,\overline{v}') \leq \varepsilon} {\mathbf 1}_{d(v',\overline{v}') \leq \varepsilon}. 
\]
Now the equalities
\begin{align*}
 \left|\frac{v+v'}{2} - v \right| & = \left|\frac{v+v'}{2} - v'\right| = \left|\frac{v-v'}{2} \right|, \\
\left|\frac{v+v'}{2}\right|^2 - |v|^2 & = -\frac{1}{2} |v|^2 + \frac{1}{2} |v'|^2 - \frac{1}{4} |v-v'|^2,
\end{align*}
allow for the following estimate of the anisotropic distance:
\begin{align*}
 d \left( v, \frac{v+v'}{2} \right)^2 & = \frac{|v-v'|^2}{4} + \frac{1}{4}\left( \left( -\frac{1}{2} |v|^2 + \frac{1}{2} |v'|^2 \right) - \frac{1}{4} |v-v'|^2  \right)^2 \\
&
 = \frac{d(v,v')^2}{4} - \frac{1}{2} \left( -\frac{1}{2} |v|^2 + \frac{1}{2} |v'|^2 \right)\frac{1}{4} |v-v'|^2  +\frac{1}{4} \frac{1}{16} |v-v'|^4 
\\
 & \leq \frac{d(v,v')^2}{4} +\frac{d(v,v')^3}{8} +\frac{d(v,v')^4}{64}.
\end{align*}
If $d(v,v') \leq 1$, 
then it follows that $d \left( v, \frac{v+v'}{2} \right) \leq \frac{3}{4} d(v,v')$.
By symmetry, it also follows that
$d \left( v', \frac{v+v'}{2} \right) \leq \frac{3}{4} d(v,v')$ as well. 
In particular, when $d(v,v') \leq \frac{10}{9} \varepsilon$, the intersection of $d(v,\overline{v}') \leq \varepsilon$ and $d(v',\overline{v}') \leq \varepsilon$ will contain the set of points $v'$ for which $d(\overline{v}', \frac{v+v'}{2} ) \leq \frac{1}{6} \varepsilon$.  Therefore, we have
\[ H(v,v') \gtrsim \frac{|\tilde B_{\frac{\varepsilon}{6}} ( \frac{v+v'}{2})|}{|\tilde B_{\varepsilon} (v')|} {\mathbf 1}_{d(v,v') \leq \frac{10}{9} \varepsilon}. \]
Since $\varepsilon \leq 1$, 
the ratio of the volumes above are uniformly bounded below (as both are comparable to $\varepsilon^{n} \ang{v'}^{-1}$).  
We have therefore established that
$
J_{\varepsilon}(f)
\ge c
J_{\frac{10}{9}\varepsilon}(f).
$
for a constant $c>0$ which is uniform in $\varepsilon$.
Iterating this inequality, it follows that, for any fixed constant $\rho > 0$
\begin{equation}
J_{\varepsilon}(f) \geq c_{\rho} J_{\varepsilon \rho}(f)
\label{mainLOWERiter}
\end{equation}
provided only that $\rho \varepsilon \leq 1$.
Now by \eqref{lowerMAINsum} and Proposition \ref{propGlower}, we have
$$
N_g(f)~  \nsm g \nsm_{L^1(\ballR)}
 \gtrsim   
 \Clower^2
  \sum_{k=k_0}^\infty 2^{k(n + 2s)}  ~ J_{\varepsilon 2^{-k}}(f).
$$
But by \eqref{mainLOWERiter}, we have
\[     J_{\varepsilon 2^{-k}}(f) \geq c_{2^{k_0} \varepsilon^{-1} } J_{2^{-k+k_0}},
 \]
uniformly provided that $k \geq k_0$, 
where the precise value of $c_{2^{k_0} \varepsilon^{-1} }$ is irrelevant since it is ultimately determined only by the dimension, $R$, and $\delta$.
Thus we have that
$$
N_g(f)~  \nsm g \nsm_{L^1(\ballR)}
 \gtrsim   
 \Clower^2
  \sum_{k=k_0}^\infty 2^{k(n + 2s)}  ~ J_{\varepsilon 2^{-k}}(f) \gtrsim  \Clower^2
  \sum_{k=k_0}^\infty 2^{k(n + 2s)}  ~ J_{2^{-k+k_0}}(f).
$$
Shifting the index of summation on the right-hand side down by $k_0$ and using the estimate \eqref{lowersumineq} with $\varepsilon = 1$ clearly establishes that the right-hand side is uniformly bounded below by $C_g^2 |f|^2_{\dot N^{s,\gamma}}$. This is
exactly \eqref{mainlower} and therefore the proof of Theorem \ref{mainLOWERthm} is complete.

\section{Entropy production estimate}
\label{sec:entropy}

In this section we prove Theorem \ref{entTHM},
which will follow from Theorem \ref{mainLOWERthm}.
We use the entropy production functional \eqref{entropyPRODf}.
From, for instance, \cite{MR1715411,MR1765272}, we split \eqref{entropyPRODf} as
$$
D(g,f) =  S(g,f) + T(g,f),
$$
where 
\begin{equation}
\begin{split}
S(g,f) & \eqdef 
  \int_{\threed}dv \int_{\threed} dv_* \int_{\sph} d\sigma ~ B (v-v_*, \sigma) ~ g_* ~
\left(f \log \frac{f}{f'} + f'  -  f  \right),
\\
T(g,f) & \eqdef 
-  \int_{\threed} dv_* ~ g_* ~
\int_{\threed}dv  \int_{\sph} d\sigma ~ B (v-v_*, \sigma) ~ \left(f'  - f  \right).
\end{split}
\notag
\end{equation}
This follows easily from the pre-post collisional change of variables.   
From the elementary inequality
$
a \log\frac{a}{b} - a + b \ge  ( \sqrt{a} - \sqrt{b} )^2,
$
see e.g. \cite{MR1765272}, we obtain that 
$$
S(g,f) \ge
  \int_{\threed}dv \int_{\threed} dv_* \int_{\sph} d\sigma ~ B (v-v_*, \sigma) ~ g_* ~
\left(\sqrt{f'}  -  \sqrt{f}  \right)^2,
$$
which immediately implies $S(g,f) \ge N_g(\sqrt{f})$ with $N_g(\sqrt{f})$ from \eqref{cancellationSPLIT}.  Now Theorem \ref{entTHM} follows easily from Theorem \ref{mainLOWERthm} and the following estimate for $T_g(f)$.
Notice that with the cancellation lemma \cite{MR1765272}, using the arguments as in Section \ref{sec:ls} with
\eqref{cancellationSPLIT}
and
\eqref{est:KGest}, it is easy to see that
$\left| T_g(f)\right| \lesssim C'  \Cupper \nsm f \nsm_{L^1_{\gamma}}$.
This establishes Theorem \ref{entTHM}.

\section*{Appendix: The dual formulation and other representations}

Finally, we derive the dual formulation from \eqref{dualOPdef} and \eqref{3dualZ}.  For this we will use the Carleman-type representation which can be found for instance in \cite[Appendix]{gsNonCutJAMS}.
Notice also that \cite[Appendix C]{09-GPV} gives a proof of Carleman-type representations.
The functions $b$ and $\Phi$ below are given by \eqref{kernelQ} and \eqref{kernelP}.  However, to derive the dual formulation it suffices to suppose that both of these functions are smooth.  The general expressions can then be deduced by approximation.  Furthermore,   we derive the ``co-plane change of variables'' from \eqref{doubleCARL}.

\subsection*{Dual Representation}
We initially suppose that $\int_{{\mathbb S}^{n-1}} d \sigma ~ |b( \ang{k, \sigma} )| < \infty$ and that the kernel $b$ has mean zero, i.e., $\int_{{\mathbb S}^{n-1}} d \sigma ~ b(\ang{k,\sigma}) = 0$. 
Then after the pre-post change of variables $(v, v_*) \to  (v', v_*')$ we can express \eqref{BoltzCOL} as
\begin{multline*}
 \ang{f,\collOPd h}= \int_{\threed} \! \! dv  \int_{\threed} \! \! dv_* \int_{{\mathbb S}^{n-1}} \! \! d \sigma
 ~ \Phi(|v-v_*|) b \left( \ang{k, \sigma} \right) g_* f \left(h' - h \right) \nonumber 
  \\
  = \int_{\threed} dv ~ \int_{\threed} dv_* ~ \int_{\sph}  d \sigma ~ \Phi(|v-v_*|) b \left( \ang{k, \sigma} \right) g_* f  h'. \nonumber 
 \end{multline*}
Now with the Carleman representation, e.g.
\cite[Appendix]{gsNonCutJAMS}, we have
 \begin{gather*}
\ang{f,\collOPd h}=
 2^{n-1} \int_{\threed} \!  \!dv_* \! \int_{\threed} \! \! dv' \!   \int_{E^{v'}_{v_*}}d\pi_{v} ~
\Phi(|v-v_*|)  
\frac{b\left(\ang{\frac{v-v_*}{|v-v_*|}, \frac{2v' - v - v_*}{|2v' - v - v_*|} } \right)}{|v' - v_*|~|v - v_*|^{n-2}} g_* f  h'. 
 \nonumber
\end{gather*}
The definitions of these notations, $E^{v'}_{v_*}$ and $d\pi_{v}$,  were given previously in the paragraph containing \eqref{dualOPdef}.
Furthermore from the identity \eqref{Cidenity} we observe that
\begin{align*}
\int_{E_{v_*}^{v'}} d\pi_{v} ~ &  b \left( \ang{\frac{v-v_*}{|v-v_*|}, \frac{2v' - v - v_*}{|2v' - v - v_*|} } \right) \frac{|v'-v_*|^{n-1}}{|v-v_*|^{2n-2}} \\
& = \int_{{\mathbb S}^{n-2}} d \sigma 
\int_0^\infty r^{n-2} \ dr \ b \left( \frac{|v'-v_*|^2 - r^2}{|v' - v_*|^2 + r^2} \right) \frac{|v' - v_*|^{n-1}}{(r^2 + |v'-v_*|^2)^{n-1}} = 0,
\end{align*}
by a change of variables to polar coordinates since $\int_{-1}^1 dt \ b(t) (1-t^2)^{\frac{n-3}{2}} = 0$ (following from the cancellation condition on $\sph$) and 
\[ \frac{d}{dr} \left[ \frac{|v'-v_*|^2 - r^2}{|v' - v_*|^2 + r^2} \right] = 
\frac{-4 r |v' - v_*|^2}{(r^2 + |v' - v_*|^2)^2}, \]
\[ \left( 1 - \left(\frac{|v'-v_*|^2 - r^2}{|v' - v_*|^2 + r^2}\right)^2 \right)^\frac{n-3}{2} = \frac{(2 r |v'-v_*|)^{n-3}}{(r^2 + |v' - v_*|^2)^{n-3}}. \]  
In particular, this implies with $\tilde{B}$ from \eqref{kernelTILDE} that
\[ 
\int_{E_{v_*}^{v'}} d \pi_v ~ \tilde{B} ~     \frac{\Phi(|v'-v_*|) |v'-v_*|^n}{\Phi(|v-v_*|) |v-v_*|^n} ~ g_*~ f'~ h' = 0. 
\]
We subtract this expression from the Carleman representation just written for 
$ \ang{f, \collOPd h}$, and use the kernel defined in \eqref{kernelTILDE} to see that \eqref{dualOPdef} holds 
with \eqref{opGlabel}.

The claim is now that this representation \eqref{dualOPdef} holds even when the mean value of the singular kernel $b(\ang{k, \sigma})$ from \eqref{kernelQ}  is not zero.  To see this claim, suppose that $b$ is integrable but without mean zero.  Then define
$$
b_\epsilon(t) \eqdef b(t) - \ind_{[1-\epsilon,1]}(t) \int_{-1}^1 dt ~ b(t) ~ (1-t^2)^{\frac{n-3}{2}} ~ 
\left( \int_{1-\epsilon}^1 dt (1-t^2)^{\frac{n-3}{2}}\right)^{-1}.
$$
As a function on ${\mathbb S}^{n-1}$, $b_\epsilon$ will clearly have a vanishing integral.  However, given arbitrary $f$, $g$ and $h$ which are Schwartz functions, it is not hard to see that
\[ 
\left| \ang{f, \collOPd h} - \ang{ f, \collOPd^\epsilon h}\right| \rightarrow 0,
\quad
\epsilon \rightarrow 0. 
\]
Above $\collOPd^\epsilon$ is the operator $\collOPd$ formed with $b_\epsilon(t)$ in place of $b(t)$.  This convergence holds 
because  cancellation guarantees that the integrand vanishes on the set defined by $\ang{k,\sigma} = 1$.  Moreover, 
an additional cutoff argument shows that the equality also holds provided that $b(t)$ satisfies \eqref{kernelQ}; 
the higher-order cancellation is preserved because $\frac{|v'-v_*|}{|v-v_*|}$ possesses radial symmetry in $v-v'$.

The  ``dual representation'' \eqref{dualOPdef} deserves its name
because if one defines
\begin{align*}
\collOPd h (v) & \eqdef \int_{\threed} dv_* \int_{{\mathbb S}^{n-1}} d \sigma ~B~ g_* \left( h' -  h \right), 
\\
\ScollOPd f (v') & \eqdef  \int_{\threed} dv_* \int_{E_{v_*}^{v'}} d \pi_{v} ~\tilde{B}~    g_*
 \left(   f   -    \frac{\Phi(v'-v_*) |v' - v_*|^n}{\Phi(v-v_*) |v-v_*|^n}  f' \right),
\end{align*} 
then 
\begin{equation}
\label{3dualZ}
\ang{ \mathcal{Q}(g, f) , h} 
=
\ang{ f , \collOPd h} 
= 
\ang{\ScollOPd f,  h }. 
\end{equation}
Note that the last inner product above represents an integration over $dv'$ whereas the first two inner products above represent integrations over $dv$.

\subsection*{The co-plane identity}
The goal of this section is to prove the co-plane change of variables from \eqref{doubleCARL}.  
It suffices to prove the formula \eqref{doubleCARL} for continuous, compactly supported $H$.  The proof follows rather directly from the coarea formula for codimension $2$, see e.g. \cite{MR0257325}.  For fixed $v$, for example, the coarea formula gives 
\[ 
\int_{\dsp_{v_*}^v} \! d \sigma_{v'} H = \lim_{\epsilon \rightarrow 0^+} \left( \frac{|v-v_*|}{2} \right)^{-(n-1)} \frac{1}{\epsilon} \int_{\threed} dv' H |v-v_*| {\mathbf 1}_{\left| \ang{v' - v, v' - v_*} \right| \leq \frac{\epsilon}{2}}, 
\]
since the sphere ${\mathbb S}^v_{v_*}$, as a function of $v'$, is precisely the zero set 
of the Lipschitz function
$ 
\ang{v' - v, v' - v_*}
$
as in \eqref{sphereUN}.
The quantity $|v-v_*|$ inside the integrand is the magnitude of the gradient of $\ang{v' - v, v' - v_*}$ on its zero set (which is required by the coarea formula) and the external term $2^{n-1} |v-v_*|^{-(n-1)}$ is a normalization factor.  Thus the left side of \eqref{doubleCARL} may be realized as a limit as $\epsilon \rightarrow 0^+$ of 
\[ 2^{2(n-1)} \frac{1}{\epsilon^2} \int_{\threed} \! dv \int_{\threed} \! dv' \int_{\threed} \! d\overline{v}' \frac{H {\mathbf 1}_{\left| \ang{v' - v, v' - v_*} \right| \leq \frac{\epsilon}{2}} {\mathbf 1}_{\left| \ang{\overline{v}' - v, \overline{v}' - \overline{v}_*} \right| \leq \frac{\epsilon}{2}} }{|v-v_*|^{n-2} |v - \overline{v}_*|^{n-2}}. \]
Now use the Fubini theorem to evaluate the integral with respect to $v$ first.  In this case the limit is supported on the intersection of the two zero sets, which is exactly the co-plane $E_{v',v_*}^{\overline{v}',\overline{v}_*}$.  The $2$-dimensional Jacobian of these two constraint functions (now with respect to $v$) is exactly $D^\frac{1}{2}$ from \eqref{quantityDdef}, so that we have
\begin{align*}
 \lim_{\epsilon \rightarrow 0^{+}}   \frac{1}{\epsilon^2} \int_{\threed} dv  \frac{H {\mathbf 1}_{\left| \ang{v' - v, v' - v_*} \right| \leq \frac{\epsilon}{2}} {\mathbf 1}_{\left| \ang{\overline{v}' - v, \overline{v}' - \overline{v}_*} \right| \leq \frac{\epsilon}{2}} }{|v-v_*|^{n-2} |v - \overline{v}_*|^{n-2}} &   \\
= \int_{E_{v',v_*}^{\overline{v}',\overline{v}_*}} \! d \pi_{v} &  \frac{H}{|v-v_*|^{n-2} |v - \overline{v}_*|^{n-2} D^{\frac{1}{2}}}. 
\end{align*}
Substituting this back into the integral over $v'$ and $\overline{v}'$ establishes \eqref{doubleCARL}.

\end{document}